\documentclass[11pt,reqno]{amsart}

\usepackage{amsmath,amsfonts,amsthm,amscd,amssymb,graphicx,mathrsfs}
\usepackage{cite}
\usepackage[utf8]{inputenc}
\usepackage{amsbsy}
\usepackage{graphicx}
\usepackage{subcaption}
\usepackage[unicode=true,bookmarks=false,breaklinks=false,pdfborder={0 0 1},colorlinks=true]{hyperref}
\usepackage{bm}
\usepackage{cleveref}
\usepackage{tikz}
\usepackage{pgfplots}
\usetikzlibrary{arrows}
\usepackage[text={6.5in,9in},centering,includefoot,foot=0.6in]{geometry}

\usepackage[all]{xy}

\numberwithin{equation}{section}
\numberwithin{figure}{section}

\usepackage{amsmath,amsthm,amsfonts,amssymb,amscd}
\usepackage[all]{xy}
\usepackage{sseq}

\usepackage{setspace}
\usepackage{todonotes}

\makeatletter
\renewcommand\subsection{\@startsection{subsection}{2}%
  \z@{-.5\linespacing\@plus-.7\linespacing}{.5\linespacing}%
  {\normalfont\bfseries}}
\makeatother

\definecolor{skyblue}{rgb}{0.85,0.85,1}

\newtheorem{theorem}{Theorem}[section]
\newtheorem{lemma}[theorem]{Lemma}
\newtheorem{prop}[theorem]{Proposition}
\newtheorem{cor}[theorem]{Corollary}

\theoremstyle{definition}

\newtheorem{define}[theorem]{Definition}

\theoremstyle{remark}
\newtheorem{rem}[theorem]{Remark}

\DeclareMathOperator{\dv}{div}

\DeclareMathOperator{\Hess}{Hess}
\DeclareMathOperator{\supp}{supp}
\DeclareMathOperator{\dom}{dom}
\DeclareMathOperator{\id}{id}

\DeclareMathOperator{\Diff}{Diff}

\DeclareMathOperator{\spn}{span}

\newcommand{\bbN}{\mathbb{N}}
\newcommand{\bbR}{\mathbb{R}}

\newcommand{\cD}{\mathcal{D}}

\newcommand{\cL}{\mathcal{L}}
\newcommand{\cP}{\mathcal{P}}
\newcommand{\cT}{\mathcal{T}}
\newcommand{\cU}{\mathcal{U}}
\newcommand{\cV}{\mathcal{V}}
\newcommand{\cW}{\mathcal{W}}

\newcommand{\pO}{\partial \Omega}

\newcommand{\lra}{\longrightarrow}
\newcommand{\p}{\partial}

\newcommand{\DP}{\Delta^{\!\p P}}

\newcommand{\Ds}{\mathcal{D}^s(N)} 
\newcommand{\DsM}{\mathcal{D}_M^s(N)} 
\newcommand{\EsM}{\mathcal{E}_P^s(N)}
\newcommand{\EsMstar}{\mathcal{E}_P^{s_*}(N)} 
\newcommand{\DsP}{\cD^s_P(N)}

\newcommand{\DtnNu}{\Lambda_{\scriptscriptstyle P,\nu}}

\newcommand{\Snu}{S_{\scriptscriptstyle P,\nu}}

\newcommand{\lap}{\lambda_{\scriptscriptstyle P}}
\newcommand{\form}{a_{_{P,\nu}}}

\newcommand{\CF}{C^\infty_0(M\backslash F; \Sigma) }
\newcommand{\wtH}{\widetilde H^{1/2}}

\allowdisplaybreaks

\begin{document}

\title{Stability of spectral partitions with corners}

\author{G. Berkolaiko}
\address{Department of
  Mathematics, Texas A\&M University, College Station, TX 77843-3368, USA}
\email{berko@math.tamu.edu}

\author{Y. Canzani}
\address{Department of Mathematics, University of North Carolina at Chapel Hill,
Phillips Hall, Chapel Hill, NC  27599, USA}
\email{canzani@email.unc.edu}

\author{G. Cox}
\address{Department of Mathematics and Statistics, Memorial University of Newfoundland, St. John's, NL A1C 5S7, Canada}
\email{gcox@mun.ca}

\author{P. Kuchment}
\address{Department of
  Mathematics, Texas A\&M University, College Station, TX 77843-3368, USA}
\email{kuchment@tamu.edu}

\author{J.L. Marzuola}
\address{Department of Mathematics, University of North Carolina at Chapel Hill,
Phillips Hall, Chapel Hill, NC  27599, USA}
\email{marzuola@math.unc.edu}

\begin{abstract}
A spectral minimal partition of a manifold is a decomposition into disjoint open sets that minimizes a spectral energy functional. While it is known that bipartite minimal partitions correspond to nodal partitions of Courant-sharp Laplacian eigenfunctions, the non-bipartite case is much more challenging. In this paper, we unify the bipartite and non-bipartite settings by defining a modified Laplacian operator and proving that the nodal partitions of its eigenfunctions are exactly the critical points of the spectral energy functional. Moreover, we prove that the Morse index of a critical point equals the nodal deficiency of the corresponding eigenfunction. Some striking consequences of our main result are: 1) in the bipartite case, every local minimum of the energy functional is in fact a \emph{global} minimum; 2) in the non-bipartite case, every local minimum of the energy functional minimizes within a certain topological class of partitions. Our results are valid for partitions with non-smooth boundaries; this introduces considerable technical challenges, which are overcome using delicate approximation arguments in the Sobolev space $H^{1/2}$.

\end{abstract}

\maketitle
\parskip=-0.5ex{\tableofcontents}
\parskip=1ex

\section{Introduction}

\subsection{Overview and main results}

For a compact, oriented manifold $M$, the \emph{spectral minimal
  $k$-partitions} are the minimizers of the functional
\begin{equation}\label{Lambda}
    \Lambda(P):=\max\limits_i \lambda_1(\Omega_i),
\end{equation}
over the \emph{partitions} $P = \{\Omega_i\}_{i=1}^k$ of $M$, namely
the collections of mutually disjoint, nonempty, open, connected
subsets $\Omega_i \subset M$.  Here $\lambda_1(\Omega_i)$ is the
lowest eigenvalue of the Dirichlet Laplacian on $\Omega_i$, and the
minimizers are sought for a given fixed $k$.  Existence of a minimizer
for any $k$ was established in \cite{ConTerVer_cvpde05}. It can be
easily seen that a minimizer must be an \emph{equipartition}:
$\lambda_1(\Omega_i) = \lambda_1(\Omega_1)$ for all $i$.

Spectral minimal partitions arise naturally 
in the study of free boundary variational problems \cite{ACF} 
and spatial segregation in the strong competition
limit of reaction--diffusion systems
\cite{ConTerVer_iumj05,ConTerVer_am05,ConTerVer_cvpde05}. They have also been used to
establish spectral gaps for ergodic Markov operators on
general measurable spaces \cite{miclo2015hyperboundedness}.  In the
wider context of optimal partitioning of manifolds and graphs,
spectral partitions have been used to approximate Cheeger cuts
\cite{osting2014minimal} and compute bounds on Cheeger constants \cite{lee2014multiway}, 
and have also been applied to problems in 
community detection \cite{beck2023uniform} and wavelet construction
\cite{beck2021nodal}. 
  Much of what is known about spectral minimal
partitions and their applications is summarized in the recent survey
\cite{bonnaillie2015nodal}.  We remark that finding spectral minimal
partitions both numerically and analytically is extremely challenging
\cite{bonnaillie2010numerical,Bogosel}.  A striking example of an open
problem is the spectral minimal 3-partition of the disk, conjectured
but not yet proven to be radial, as shown in \Cref{fig:MS}.

In this paper, we are primarily interested in the connection between
spectral minimal partitions and nodal properties of
Laplacian eigenfunctions, an area with rich history and much recent
activity (for a small subset of the work in this area, see \cite{beck2021nodal,beck2024nodal,lyons2023nodal,
  P56,DonFef_jga92,NazSod_ajm09,LogMal_otaa18,Log_am18,CanSar_cpam19,JunZel_aif20,GanMckMohSri_prep21,AloBanBer_em22}).
This connection was first discovered in the seminal
articles \cite{ConTerVer_cvpde05,HHOT} and further
developed in
\cite{BanBerRazSmi_cmp12,BKS12,helffer2013magnetic,BCCM2,BCCM3}.  In
particular, it turns out to be beneficial to alter the scope of the
problem, discussing \emph{all} critical points of the functional $\Lambda(P)$\,---\,as 
opposed to just the minima\,---\,while restricting the search
to the space of equipartitions.  It became increasingly clear that the
critical points are generated by the nodal sets of
Laplacian\footnote{When the partition is not bipartite (and thus
  cannot possibly be generated by the sign changes of a Laplacian
  eigenfunction), the Laplacian has to be suitably modified, as described 
  in \eqref{eq:PLap} and \Cref{sec:DeltaP}.} eigenfunctions, with the
stability of said critical points governed by the position of the
corresponding eigenvalue in the spectrum.  So far, this correspondence
has been established only for partitions with smooth internal
boundaries $\partial \Omega_i \backslash \partial M$.  In the present
work we finally remove the assumption of smoothness, establishing the
above correspondence for all critical points of $\Lambda$, including
partitions with triple intersection points (as in \Cref{fig:MS}), which are conjectured to be generic for large
$k$.  As a further benefit, whenever a given critical partition is
not minimal, our technique suggests ways to perturb the partition to
move towards a minimum.

To start introducing our results, we collect the relevant terminology
in the following definition. For the rest of the paper we focus on 
the two-dimensional case, so our results are valid for partitions of 
compact surfaces and planar domains.

\begin{define}
\label{def:regular}
A $k$-\emph{partition} of a compact, oriented, connected surface $M$ is a collection $P = \{\Omega_i\}_{i=1}^k$ of mutually
disjoint, nonempty, open, connected subsets $\Omega_i \subset M$. If each $\overline\Omega_i$ is a manifold with piecewise smooth
boundary, $\operatorname{int}\overline\Omega_i = \Omega_i$ and $M = \overline{\cup_i\Omega_i}$, then $P$ is a \emph{partition with corners}. If $\lambda_1(\Omega_i)$
does not depend on $i$, then $P$ is an \emph{equipartition}. A partition is \emph{bipartite} if one can assign to each $\Omega_i$
one of two colors so that any two neighbors have different colors.
\end{define}

\begin{figure}
\begin{tikzpicture}[ scale=0.8]
	\draw[thick,dashed] (2,0) arc[radius=2, start angle=0, end angle=360];
	\draw[very thick] (0,0) -- ({2*cos(90)},{2*sin(90)});
	\draw[very thick] (0,0) -- ({2*cos(210)},{2*sin(210)});
	\draw[very thick] (0,0) -- ({2*cos(330)},{2*sin(330)});
\end{tikzpicture}
\caption{The so-called ``Mercedes star" is conjectured to be the minimal 3-partition of the disk.}
\label{fig:MS}
\end{figure}
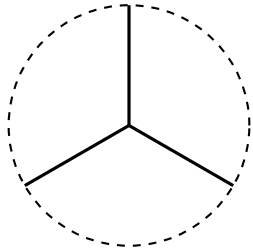

An important example is the \emph{nodal partition} of
an eigenfunction $\psi$ of the Dirichlet Laplacian on $M$. This is the partition whose subdomains are the connected components of the set $\{\psi \neq 0\}$.
Nodal partitions of eigenfunctions are bipartite equipartitions. If $\lambda$ denotes the 
eigenvalue for $\psi$, then each subdomain has $\lambda_1(\Omega_i) = \lambda$, therefore 
the partition energy is $\Lambda(P) = \lambda$. Defining the \emph{spectral position} to be $\ell(P) = \min\{ n : 
\lambda_n = \lambda\}$, where $\{\lambda_n\}_{n=1}^\infty$ are the eigenvalues of the Dirichlet Laplacian on $M$, 
we then define the \emph{nodal deficiency} 
\begin{equation}
	\delta(P) = \ell(P) - k,
\end{equation}
where $k$ 
is the number of components in $P$. Courant's nodal domain theorem states that $\delta(P) \geq 0$. We say that the eigenfunction $\psi$ is \emph{Courant sharp} if its nodal partition has $\delta(P) = 0$.

One of the difficulties in finding spectral minimal partitions is the
need to find the \emph{global} minimum.  A striking special case
of our much more general main result (Theorem \ref{thm:equality})
shows that for bipartite partitions it suffices to check that $P$
\emph{locally minimizes} $\Lambda$, in the sense that
\begin{equation}
  \Lambda(P)\leq \Lambda(\varphi(P))
\end{equation}
for all diffeomorphisms $\varphi$ close to identity (in a sense that will be made precise later).

\begin{theorem}
  \label{thm:bi}
  Let $P$ be a bipartite $k$-partition with corners. The following are
  equivalent:
  \begin{enumerate}
  \item $P$ is the nodal partition of a Courant-sharp Laplacian eigenfunction,
  \item $P$ is a minimal partition,
  \item $P$ locally minimizes $\Lambda$.
  \end{enumerate}
\end{theorem}
In particular, every bipartite local minimum of $\Lambda$ is in fact a global minimum. The significance of this result is that the hypothesis only involves comparing $P$ to partitions that are nearby (and in particular have the same topology), but the conclusion is that $P$ minimizes energy compared to \emph{all} $k$-partitions, regardless of proximity or topology; see \Cref{fig:square4} for an illustration.

The equivalence (1) $\Leftrightarrow$ (2) was shown in \cite{HHOT}, and the implication (2) $\Rightarrow$ (3) is trivial. In this paper we prove (3) $\Rightarrow$ (1), and as a consequence obtain a new proof that (2) $\Rightarrow$ (1), which is nontrivial.

In fact, our approach gives a much more general result than
\Cref{thm:bi}, which only dealt with minimal partitions. 
Given an equipartition $P$, we consider the set of diffeomorphisms $\varphi \colon M \to M$ such that $\varphi(P)$ is also an equipartition, and restrict the energy functional $\Lambda$ to it. That is, we work with the functional $\lap(\varphi) = \Lambda\big(\varphi(P)\big)$; see equation~\eqref{eq:Es_def} and \Cref{thm:submanifold} for precise definitions.
This is a smooth function, and thus
the notions of critical points (not just minima) and
their Hessians make sense. Starting with the bipartite case, we 
completely characterize the critical points of $\lap$ and their Morse 
indices, in terms of Laplacian eigenfunctions and their nodal deficiency.
We say that an equipartition $P$ is a \emph{critical partition for $\lap$} if $\varphi=\id$ is a critical point of $\lap$.

\begin{theorem}
\label{thm:equality BP}
A bipartite partition $P$ is a critical partition for $\lap$  if and only if it is the nodal partition of an eigenfunction of $-\Delta$, in which case it has Morse index $n_-(\Hess \lap) = \delta(P)$.
\end{theorem}

Here, the \emph{Morse index}, $n_-(\Hess \lap)$, is the maximal dimension
of a subspace on which the symmetric bilinear form $\Hess\lap$ is negative definite.

A special case of \Cref{thm:equality} was obtained previously in \cite[Theorem~11]{BKS12}, for smooth, bipartite equipartitions corresponding to simple eigenfunctions of the Laplacian. To the best of our knowledge, this was the first explicit formula obtained for the nodal deficiency. Another formula was given in \cite{CJM2} in terms of a two-sided Dirichlet-to-Neumann map defined on the nodal set; see \Cref{thm:BCHS}. The present work arose from our desire to unify these two approaches to nodal deficiency, and in particular, to understand the relationship between the Hessian of $\lap$ and the two-sided Dirichlet-to-Neumann map. This relationship is made explicit in \Cref{thm:Hess1}, which is one of the main ingredients in the proof of \Cref{thm:equality BP} and its non-bipartite generalization, \Cref{thm:equality}.

The general (not necessarily bipartite) case is much more subtle and requires additional terminology. Suppose $P$ is a $k$-partition of $M$ with \emph{boundary set} $\p P = \cup_i \overline{\pO_i\backslash\p M}$. We say that a set $\Gamma$ is \emph{homologous to $\p P$} if the closure of the symmetric difference, $\overline{\p P \triangle \Gamma}$, is the boundary set of a bipartite partition. We then define 
\begin{equation}
\label{Pkdef}
	\cP_k(P) = \big\{ \tilde P : \tilde P \text{ is a $k$-partition and }\p \tilde P \text{ contains a subset homologous to } \p P \big\}.
\end{equation}
For instance, if $P$ is the 3-partition of the disk shown in \Cref{fig:MS}, then $\cP_3(P)$ consists of all 3-partitions whose boundary set contains a curve from the origin to the boundary of the disk; see \Cref{ssec:hom} and \cite{BCCM3} for details.

We also define an operator $-\DP$, called the \emph{partition Laplacian}, that acts as the Laplacian on the space of functions satisfying Dirichlet boundary conditions on $\p M$ and anti-continuity conditions on the boundary set of $P$. That is, the restrictions $u_i = u\big|_{\Omega_i}$ satisfy
\begin{equation}
\label{eq:PLap}
	u_i\big|_{\pO_i \cap \pO_j} = - u_j\big|_{\pO_i \cap \pO_j}, \qquad \frac{\p u_i}{\p \nu_i}\bigg|_{\pO_i \cap \pO_j} = \frac{\p u_j}{\p \nu_j}\bigg|_{\pO_i \cap \pO_j}
\end{equation}
whenever $\pO_i$ intersects $\pO_j$ for some $i \neq j$. The partition Laplacian is unitarily equivalent to the Laplacian when $P$ is bipartite; see \Cref{sec:DeltaP} for details.
The non-bipartite generalization of Theorem \ref{thm:bi} can now be stated as follows.

\begin{theorem}
\label{thm:nonbi}
Let $P$ be a $k$-partition with corners. The following are equivalent:
\begin{enumerate}
	\item $P$ is the nodal partition of a Courant sharp eigenfunction for $-\DP$,
	\item $\Lambda(P) = \inf \{\Lambda(\tilde P) :  \tilde P \in \cP_k(P) \}$,
	\item $P$ is a critical partition for $\lap$ with $n_-(\Hess\lap) = 0$.
\end{enumerate}
\end{theorem}

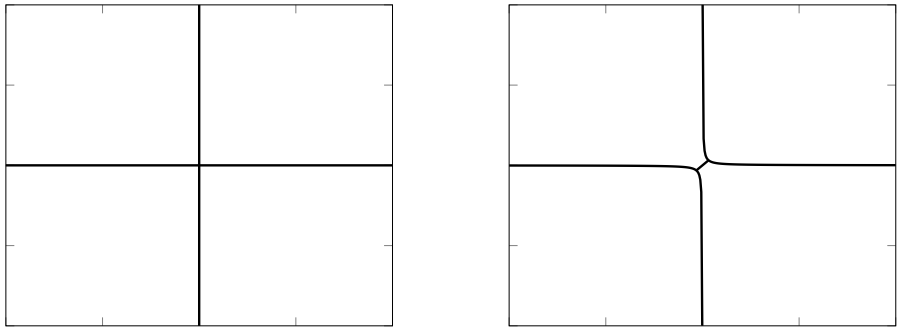
\begin{figure}[t]
\begin{center}
\begin{tikzpicture}[scale=0.75]
\begin{axis}[xmin=-1,ymin=-1,xmax=1,ymax=1, samples=50, yticklabels=\empty, xticklabels=\empty]
	\addplot[domain=-1:1, very thick] (x,0);
	\addplot[domain=-1:1, very thick] (0,x);
\end{axis}
\end{tikzpicture}
\hspace{1cm}
\begin{tikzpicture}[scale=0.75]
\begin{axis}[xmin=-1, ymin=-1, xmax=1, ymax=1, samples=200, yticklabels=\empty, xticklabels=\empty]
	\addplot[domain=-1:-0.001, very thick] (x, 0.001/x);
	\addplot[domain=0.001:1, very thick] (x, 0.001/x);
	\addplot[domain=-.03:0.03, very thick] (x,x);
\end{axis}
\end{tikzpicture}
\caption{$P$, the partition of the square on the left, locally minimizes $\Lambda$. The partition on the right, $\tilde P$, is not included in this local comparison, as it is not of the form $\varphi(P)$ for any diffeomorphism $\varphi$ and hence is not close to $P$. Nonetheless, \Cref{thm:bi} guarantees that $\Lambda(P) \leq \Lambda(\tilde P)$.}
\label{fig:square4}
\end{center}
\end{figure}

The proof that (2) $\Rightarrow$ (3) is more delicate in the non-bipartite case, since (2) just says that $P$ is minimal within the set $\cP_k(P)$, which does not contain an open neighborhood of $P$. We will resolve this with a technical approximation result, \Cref{thm:dense}, that allows us to calculate the Morse index of $\Hess\lap$ just using deformations that fix a neighborhood of each corner point.

In contrast to \Cref{thm:bi}, we are not able to conclude that a
locally minimal (non-bipartite) partition is a global minimum.  In
fact, numerical examples show that there exist local minima which are
not global.  However, we conjecture that locally minimal
$k$-partitions $P_1$ and $P_2$ with $\Lambda(P_1) \neq \Lambda(P_2)$
necessarily have different topology.

Using again the restriction $\lap$ of the energy functional $\Lambda$ to the set of
\emph{equipartitions}, we are able to obtain a not-necessarily-bipartite analog of Theorem \ref{thm:equality BP}, to characterize the critical points and compute their Morse indices in terms of the \emph{deficiency}
\begin{equation}
\label{def:delta}
	\delta(P) = n_-\big(-\DP - \Lambda(P)\big) + 1 - k,
\end{equation}
which is nonnegative by \cite[Theorem~1.7]{BCHS}. If $\delta(P) = 0$, then $P$ is said to be \emph{Courant sharp}. An equivalent formula for $\delta(P)$ will be given in \Cref{rem:delta}.

\begin{theorem}
\label{thm:equality}
$P$ is a critical partition for $\lap$ if and only if it is the nodal partition of an eigenfunction of $-\DP$, in which case it has Morse index $n_-(\Hess \lap) = \delta(P)$.
\end{theorem}

We therefore get information about
saddle points of the energy functional, in addition to local minima. We will use this in
\Cref{sec:examples} to construct energy-decreasing perturbations of
non-minimal partitions of the disk and rectangle.

\begin{rem}
\label{rem:delta}
If $P$ is a critical (and hence nodal) partition, we can define its spectral position $\ell(P) = \min\{n : \lambda_n(\p P) = \Lambda(P)\}$, where $\{\lambda_n(\p P)\}_{n=1}^\infty$ denote the eigenvalues of $-\DP$, written in increasing order and repeated according to their multiplicity. It follows that $\ell(P) = n_-\big(-\DP - \Lambda(P)\big) + 1$, so we can equivalently write the deficiency as $\delta(P) = \ell(P) - k$, which is the definition give in \cite{BCHS}.
\end{rem}

In the remainder of the introduction, we give precise definitions and describe the ideas in the proof of \Cref{thm:equality}. Along the way we highlight the relevant literature, explaining the developments that led to the above theorems.

\subsection{Definitions and further results}
\label{sec:further}
If $M$ has a nonempty boundary, it is convenient to assume that it is contained in a larger manifold, which we can take to be closed (compact and without boundary). We therefore let $(N,g)$ be a closed, oriented Riemannian surface and $M \subset N$ a compact subset with piecewise smooth (or empty) boundary.

It was shown in \cite{BKS12} that the space of generic partitions, which by definition do not have corners, has the structure of a smooth manifold. When the partitions are allowed to have corners, it is easier to consider $\Lambda$ as a function on the space of diffeomorphisms, mapping $\varphi \mapsto \Lambda(\varphi(P))$, where $P = \{\Omega_i\}$ is a fixed ``reference partition" and $\varphi(P)$ denotes the partition $\{\varphi(\Omega_i)\}$. This simplifies the analysis at the expense of introducing some degeneracy into the energy functional, since there are many nontrivial diffeomorphisms that map $P$ to itself and hence do not change the energy. In terms of the Hessian of the energy functional, this degeneracy affects the nullity but not the Morse index; see \Cref{sec:null} for further discussion.

Instead of working with diffeomorphisms on $M$, which may have corners, we consider diffeomorphisms on the ambient space $N$ that leave $M$ invariant. Let $\Ds$ be the set of $H^s$ diffeomorphisms on $N$ for some $s>2$. This is a smooth ($C^\infty$) Hilbert manifold, as we recall in \Cref{sec:manifold}, and the tangent space at the identity is the set $\mathfrak{X}^{s}(N)$ of $H^s$ vector fields on $N$. We then define
\begin{equation}
	\DsM = \big\{ \varphi \in \Ds : \varphi(M) = M \big\}.
\end{equation}
In \Cref{ssec:submanifold} we will show that $\DsM$ is a smooth submanifold of $\Ds$, with $T_{\id} \DsM$ consisting of all vector fields in $\mathfrak{X}^{s}(N)$ that are tangent to the smooth part of $\p M$.

The map $\varphi \mapsto \Lambda(\varphi(P))$ is not a smooth function on this manifold\,---\,while the individual eigenvalues $\lambda_1(\varphi(\Omega_i))$ vary smoothly when the partition is deformed, their maximum is only Lipschitz. Fortunately, it is known that a minimal partition must be an \emph{equipartition}. We thus assume that our reference partition $P$ is an equipartition and define
\begin{equation}
  \label{eq:Es_def}
	\EsM = \big\{ \varphi \in \DsM : \varphi(P) \text{ is an equipartition of $M$} \big\}.
\end{equation}
The following result says that $\EsM$ is a convenient setting for our variational analysis.

\begin{theorem}
\label{thm:submanifold}
If $P$ is an equipartition with corners and $s>2$, then $\EsM \subset \DsM$ is a smooth submanifold, with
\[	T_{\id} \EsM = \left\{ X \in T_{\id} \DsM : \int_{\pO_1} (X \cdot \nu_1) \left(\frac{\p \psi_1}{\p \nu_1}\right)^2 d\mu = 
	\cdots = \int_{\pO_k} (X \cdot \nu_k)\left(\frac{\p \psi_k}{\p \nu_k}\right)^2 d\mu \right\},
\]
where $\psi_i$ is the $L^2(\Omega_i)$-normalized groundstate on $\Omega_i$ and $\nu_i$ is the outward unit normal. Moreover, the function
\begin{equation}
\label{def:lambda}
	\lap \colon \EsM  \lra \bbR, \qquad \lap(\varphi) = \Lambda\big(\varphi(P)\big)
\end{equation}
is smooth.
\end{theorem}

\Cref{thm:submanifold} allows us to talk about critical points of $\lap$. Since $\eta$ is a critical point of $\lap$ if and only if the identity is a critical point of $\lambda_{\scriptscriptstyle \eta(P)}$, it suffices to differentiate $\lap$ at the identity. We thus say that $P$ is a \textit{critical partition} if it is an equipartition with corners and the identity is a critical point of $\lap$.

\begin{theorem}
\label{thm:critical}
$P$ is a critical partition if and only if there exist nonzero real numbers $a_1, \ldots, a_k$ such that $a_1^2 + \cdots + a_k^2 = 1$ and 
\begin{equation}
\label{eq:normals}
	\left| a_i \frac{\p \psi_i}{\p \nu_i} \right| = \left| a_j \frac{\p \psi_j}{\p \nu_j} \right|
\end{equation}
on each $\Sigma_i \cap \Sigma_j$, where $\Sigma_i = \overline{\pO_i \backslash \p M}$ is the ``internal boundary" of $\Omega_i$.
\end{theorem}

\begin{rem}
The function $\lap$ depends on the choice of $s$ but we will usually suppress this dependence. If the identity is a critical point of $\lap \colon \EsMstar  \to \bbR$ for some $s_*>2$, then \eqref{eq:normals} holds and it follows that the identity is a critical point of $\lap \colon \EsM  \to \bbR$ for \emph{every} $s>2$. In other words, the notion of ``critical partition" does not depend on the choice of $s>2$.
\end{rem}

An immediate consequence of \Cref{thm:critical} is that $P$ is critical if and only if it is the nodal partition of an eigenfunction of $-\DP$. This generalizes \cite[Theorem~9]{BKS12} to the non-smooth, non-bipartite case. For convenience, we will often denote the boundary set $\p P$ by $\Sigma$, so that $\Sigma = \cup_i \Sigma_i$.

Another consequence of \Cref{thm:critical} is that we can define a ``weight function" $\rho \colon \Sigma \to \bbR$ by
\begin{equation}
\label{rhodef}
	\rho\big|_{\Sigma_i} = \left| a_i \frac{\p \psi_i}{\p \nu_i} \right|
\end{equation}
for each $i$. This will appear in \Cref{thm:Hess1} when we compute $\Hess\lap$. If the partition has corners, the weight $\rho$ will vanish there. This causes technical difficulties in the proof of \Cref{thm:equality}, which we overcome by a delicate approximation argument in \Cref{sec:comparing}.

Having characterized the critical points of $\lap$, we next compute the second derivative. The proof of \Cref{thm:equality} is based on an explicit formula for $\Hess\lap$ in terms of a weighted, two-sided version of the Dirichlet-to-Neumann map, $\DtnNu $, which we now describe.

First, we choose a piecewise smooth unit vector field $\nu$ along $\Sigma$ that is normal to the smooth part of $\Sigma$. (The one-sided limits of $\nu$ therefore exist at each corner, but need not coincide.) For each $i$ we define a function $\chi_i = \nu \cdot\nu_i \colon \Sigma_i \to \{\pm1\}$ that is constant along each smooth segment of $\Sigma_i$, and let
\begin{equation}
	\Snu := \left\{ f \in L^2(\Sigma) : \int_{\Sigma_i} \chi_i f \frac{\p\psi_i}{\p\nu_i}\,d\mu = 0 \text{ for each $i$} \right\}.
\end{equation}
For sufficiently smooth $f \in \Snu$ we have $\DtnNu f = \Pi_{\Snu}\big(\p_\nu u^f\big)$, where $\Pi_{\Snu}$ is the $L^2(\Sigma)$-orthogonal projection onto $\Snu$, $\p_\nu u^f$ is a function on $\Sigma$ satisfying
\begin{equation}
	\p_\nu u^f\big|_{\Sigma_i \cap \Sigma_j} = \chi_i \frac{\p u_i^f}{\p \nu_i} + \chi_j \frac{\p u_j^f}{\p \nu_j}
\end{equation}
for all $i \neq j$, and $u_i^f$ is any solution to the boundary value problem
\begin{equation}
\label{uiequation0}
	\Delta u + \Lambda(P) u = 0, \qquad u\big|_{\Sigma_i} = \chi_i f, \qquad u\big|_{\pO_i \setminus \Sigma_i} = 0.
\end{equation}
A precise definition will be given in \Cref{sec:DtN}. For now we just mention that $\DtnNu$ is a densely defined, self-adjoint operator on $\Snu$, generated by a symmetric bilinear form $\form$.

We are now ready to state our formula for the Hessian. Given a vector field $X$ and a choice of $\nu$ as described above, we will write the normal component $X\big|_\Sigma \cdot \nu$ simply as $X \cdot \nu$.

\begin{theorem}
\label{thm:Hess1}
Fix $s>3$ and let $P$ be a critical partition. If $\nu$ is a smooth unit normal vector field along the smooth part of $\Sigma$, then
\begin{equation}
\label{Hess1}
	\Hess \lap(\id)(X_1,X_2) = 2\form\big(\rho(X_1\cdot\nu), \rho(X_2 \cdot\nu) \big)
\end{equation}
for all $X_1, X_2 \in T_{\id} \EsM$.
\end{theorem}

This is similar to the formula for the Hessian obtained in \cite{BCCM2} for generic partitions, but there are some important differences:
\begin{enumerate}
	\item It is not assumed that $X_1$ and $X_2$ are normal to the smooth part of $\Sigma$. (If they were, they would necessarily vanish at the corner points, and hence could not describe a deformation that moves these points.)
	\item The right-hand side is written in terms of the bilinear form $\form$ instead of the corresponding self-adjoint operator $\DtnNu$. (The fact that $P$ is a critical partition guarantees that $\rho(X\cdot\nu)$ is in the form domain for any $X \in T_{\id} \EsM$, but it is not clear that it is in the operator domain, which is strictly smaller, due to the loss of regularity of $\psi_i$, and hence of $\rho$, at the corners.)
\end{enumerate}

The same two-sided Dirichlet-to-Neumann map\footnote{The operator $\DtnNu$ depends on the choice of unit normal $\nu$, but its index does not.} also encodes the deficiency of the partition $P$.

\begin{theorem}\cite[Theorem~1.7]{BCHS}
\label{thm:BCHS}
If $P$ is a critical partition, then $n_-(\DtnNu) = \delta(P)$.
\end{theorem}

The theorem in this level of generality is due to \cite{BCHS}. Earlier versions for the bipartite cases appeared in \cite{BCM19,CJM2} and the non-bipartite case was considered in \cite{helffer2021spectral}.

We will prove \Cref{thm:equality} by comparing \Cref{thm:BCHS,thm:Hess1}. Specifically, we will compare the Morse indices of $\Hess\lap$ and $\DtnNu$. We know from \eqref{Hess1} that $\Hess\lap$ and $\form$ are related by the map $X \mapsto \rho(X\cdot\nu)$ from $T_{\id} \EsM$ to $\dom(\form)$, but this map is neither injective nor surjective.

The lack of injectivity stems from the obvious fact that any vector field $X$ that is tangent to $\Sigma$ does not move the partition $P$, and hence leaves the energy unchanged. (The resulting degeneracy is therefore a consequence of $\lap$ being defined on the space of diffeomorphisms, rather than partitions.) This affects the nullity of $\Hess\lap$ but not the Morse index; see \Cref{sec:null} for further discussion.

The lack of surjectivity is more subtle. If the partition has corners, then the weight function $\rho$ will vanish there, since the normal derivative of $\psi_i$ vanishes at each corner. It follows that $X \mapsto \rho(X\cdot\nu)$ is not surjective, since  every function in its range vanishes at the corners.  Nonetheless, we are able to show that the range is dense in a suitable norm, and conclude the following.

\begin{theorem}
\label{thm:index}
If $P$ is a critical partition, then $n_-(\Hess \lap) = n_-(\DtnNu)$.
\end{theorem}

In fact, we will see that any $f \in \dom(\form)$ can be approximated by $\rho(X\cdot\nu)$ where $X$ is smooth and vanishes in a neighborhood of each corner point. This means the deficiency $\delta(P)$ is equal to the Morse index of $\Hess\lap$ \emph{restricted to the space of smooth deformations that do not move the partition corners.} This will play a crucial role in the proof of \Cref{thm:nonbi}, since the class of partitions $\cP_k(P)$ is invariant under such deformations.

Our approximation result, \Cref{thm:dense}, has an interesting geometric manifestation. We will see examples in \Cref{sec:examples} where the eigenfunction for the most negative eigenvalue of $\DtnNu$ does not vanish at a corner point, and hence is not of the form $\rho(X\cdot\nu)$ for any $X \in T_{\id} \EsM$. Intuitively, this singular vector field corresponds to a ``deformation" that decreases the energy by changing the partition topology. Such ``deformations" do not remain in the space $\EsM$, but it turns out they can be \emph{approximated by} deformations in $\EsM$ that decrease the partition energy. In \Cref{sec:examples} we illustrate these approximating deformations and observe that they appear ``close" (in a sense we do not make precise here) to ``deformations" that change the partition topology.

In other words, the manifold $\EsM$ is rich enough that the Morse index of $\Hess\lap$ recovers the deficiency $\delta(P)$ without taking into account non-smooth (i.e., topology changing) deformations.

\subsection*{Outline of paper and notation}
In \Cref{sec:operators} we define the self-adjoint operators $\DtnNu$ and $-\DP$ and recall the notion of homology for partition boundary sets that was introduced in \cite{BCCM3}. In \Cref{sec:manifold} we describe the manifold $\Ds$ of diffeomorphisms on $N$, as well as the submanifolds $\DsM$ and $\DsP$.  In \Cref{sec:variation} we obtain formulas for the first and second variation of a Dirichlet eigenvalue when the domain is varied, assuming sufficient regularity of the perturbation but allowing the domain to have corners. These formulas allow us to characterize critical points of $\lap$ and compute its Hessian, which we do in \Cref{sec:Hessian}. In \Cref{sec:comparing} we use this formula for the Hessian to prove \Cref{thm:index}. As noted above, this is nontrivial and requires some delicate approximation arguments. We then use this theorem in \Cref{sec:LG} to prove \Cref{thm:bi,thm:nonbi}. Finally, in \Cref{sec:examples} we use our machinery to study non-bipartite partitions of the rectangle and disk.

For convenience we list the main spaces and objects arising in this paper.

\begin{align*}
	(N,g) & \quad\text{ a closed, oriented Riemannian surface} \\
	M & \quad\text{ a compact surface with piecewise smooth boundary, contained in $N$} \\
	M^o & \quad\text{ the interior of $M$} \\
	P = \{\Omega_i\} & \quad \text{ a partition of $M$} \\
	\Sigma = \p P & \quad \text{ the boundary set of $P$} \\
	\Sigma_i & \quad \text{ the ``internal boundary" of $\Omega_i$} \\
	-\DP & \quad\text{ the partition Laplacian associated to $P$} \\
	\delta(P) & \quad\text{ the deficiency of $P$} \\
	\mathfrak X^s(N) & \quad\text{ the space of $H^s$ vector fields on $N$} \\
	\Ds & \quad\text{ the set of $H^s$ diffeomorphisms of $N$} \\
	\DsM & \quad\text{ the diffeomorphisms in $\Ds$ that leave $M$ invariant} \\
	\DsP & \quad\text{ the diffeomorphisms in $\DsM$ that leave $P$ invariant} \\
	\EsM & \quad\text{ the diffeomorphisms in $\DsM$ that map $P$ to an equipartition} \\
	\lap & \quad\text{ the partition energy functional on $\EsM$} \\
	\nu & \quad\text{ a piecewise smooth unit vector field, normal to the smooth part of $\Sigma$}\\
	\DtnNu & \quad\text{ the weighted, two-sided Dirichlet-to-Neumann map} \\
	\form & \quad\text{ the bilinear form that generates $\DtnNu$} \\
	\rho & \quad\text{ the weight function on $\Sigma$} \\
\end{align*}


\section{Preliminaries}
\label{sec:operators}
\subsection{The two-sided Dirichlet-to-Neumann map}
\label{sec:DtN}

In this section we review the construction of the two-sided Dirichlet-to-Neumann map for $\Delta + \lambda_*$,  following the notation and presentation of \cite{BCHS}. We assume throughout this section that $P = \{\Omega_i\}$ is an equipartition with energy $\lambda_*$. This means $\lambda_*$ is the smallest eigenvalue of the Dirichlet Laplacian on each $\Omega_i$, with corresponding eigenfunction denoted $\psi_i$.

For each $i$ we define $\Sigma_i = \overline{\pO_i \backslash \p M}$ to be the ``internal boundary" of $\Omega_i$. Given a choice of $\nu$, we define a function $\chi_i = \nu \cdot\nu_i \colon \Sigma_i \to \{\pm1\}$ that is constant along each smooth segment of $\Sigma_i$.

Following the notation of \cite{M00}, we let
\begin{equation}
	\wtH(\Sigma_i) = \big\{ f \in L^2(\Sigma_i) : E_i f \in H^{1/2}(\pO_i) \big\},
\end{equation}
where $E_i \colon L^2(\Sigma_i) \to L^2(\pO_i)$  denotes extension by zero. This is complete with respect to the inner product $\left< f, g \right>_{\wtH(\Sigma_i)} = \left< E_i f, E_i g\right>_{H^{1/2}(\pO_i)}$.

\begin{rem}
For comparison, the space $H^{1/2}(\Sigma_i)$ consists of all $f \in L^2(\Sigma_i)$ that can be extended (but not necessarily by zero) to a function in $H^{1/2}(\pO_i)$. This means $\wtH(\Sigma_i)$ is a proper subset of $H^{1/2}(\Sigma_i)$ if $\Sigma_i \neq \pO_i$. An easy example is a nonzero constant function on $\Sigma_i$; it can obviously be extended to a function in $H^{1/2}(\pO_i)$, but its extension by zero is not in  $H^{1/2}(\pO_i)$.
\end{rem}

We then define
\begin{equation}
\label{def:wtHSigma}
	\wtH(\Sigma) := \left\{ f \in L^2(\Sigma) : \chi_i f\big|_{\Sigma_i} \in \wtH(\Sigma_i) \text{ for each $i$}\right\},
\end{equation}
which is is easily seen to be complete with respect to the inner product
\begin{equation}
	\left<f, g \right>_{\wtH(\Sigma)} = \sum_{i=1}^k \left< \chi_i f, \chi_i g \right>_{\wtH(\Sigma_i)}.
\end{equation}
Next, we define a closed subspace of $L^2(\Sigma)$ by
\begin{equation}
	\Snu := \left\{ f \in L^2(\Sigma) : \int_{\Sigma_i} \chi_i f \frac{\p\psi_i}{\p\nu_i}\,d\mu = 0 \text{ for each $i$} \right\}.
\end{equation}
The two-sided Dirichlet-to-Neumann map $\DtnNu$ will be constructed as the selfadjoint operator generated by a closed, lower semibounded bilinear form on $\wtH(\Sigma) \cap \Snu$.

Given $f \in \wtH(\Sigma) \cap \Snu$, consider the boundary value problem
\begin{equation}
\label{uiequation}
	\Delta u_i + \lambda_* u_i = 0, \qquad u_i\big|_{\Sigma_i} = \chi_i f, \qquad u_i\big|_{\pO_i \setminus \Sigma_i} = 0
\end{equation}
on $\Omega_i$. The boundary conditions are equivalent to $u_i\big|_{\pO_i} = E_i(\chi_i f)$. Since $f \in \wtH(\Sigma) \cap \Snu$, we have $E_i(\chi_i f) \in H^{1/2}(\pO_i)$ and
\[
	\int_{\pO_i} E_i(\chi_i f) \frac{\p\psi_i}{\p\nu_i}\,d\mu = \int_{\Sigma_i} \chi_i f \frac{\p\psi_i}{\p\nu_i}\,d\mu = 0,
\]
so there is a weak solution $u_i \in H^1(\Omega_i)$ to \eqref{uiequation}. Moreover, there is a unique solution for which $\int_{\Omega_i} u_i \psi_i \,dV = 0$. Denoting this solution by $u_i^f$, we then define
\begin{equation}
\label{aformdef}
	\form(f,h) = \sum_{i=1}^k \int_{\Omega_i} \big(\big\langle\nabla u_i^f, \nabla u_i^h\big\rangle - \lambda_* u_i^f u_i^h\big)\,dV
\end{equation}
for $f,h \in \dom(\form) = \wtH(\Sigma) \cap \Snu$.

It was shown in \cite[Section~3]{BCHS} that the form $\form$ is densely defined, closed and lower semibounded. In particular, we note the bound
\begin{equation}
\label{abound}
	|\form(f,h)| \leq C \|f\|_{\wtH(\Sigma)} \| h\|_{\wtH(\Sigma)}
\end{equation}
for all $f,h \in \wtH(\Sigma) \cap \Snu$, which will be used below in \Cref{sec:comparing}.

The form $\form$ therefore generates a selfadjoint operator $\DtnNu$ such that $\form(f,h) = \left<\DtnNu f, h\right>_{L^2(\Sigma)}$ for all $f \in \dom(\DtnNu)$ and $h \in \wtH(\Sigma) \cap \Snu$. To describe $\DtnNu$ and its domain, we introduce the two-sided normal derivative $\p_\nu u^f$ on $\Sigma$. This is an element of the dual space $\wtH(\Sigma)^*$, but for sufficiently smooth $u^f$ it is a function on $\Sigma$ satisfying
\begin{equation}
	\p_\nu u^f\big|_{\Sigma_i \cap \Sigma_j} = \chi_i \frac{\p u_i^f}{\p \nu_i} + \chi_j \frac{\p u_j^f}{\p \nu_j}
\end{equation}
for all $i \neq j$.

From \cite[Theorem~3.1]{BCHS} we know that $\dom(\DtnNu) = \big\{f \in \wtH(\Sigma) \cap \Snu : \p_\nu u^f \in L^2(\Sigma) \big\}$, and for any such $f$ we have 
\begin{equation}
	\DtnNu f = \Pi_{\Snu}\big(\p_\nu u^f\big),
\end{equation}
where $\Pi_{\Snu}$ is the $L^2(\Sigma)$-orthogonal projection onto $\Snu$.

\subsection{The partition Laplacian}
\label{sec:DeltaP}

Let $P = \{\Omega_i\}$ be a partition with corners. For any function $u$ on $M$, we define its restriction $u_i = u\big|_{\Omega_i}$ for each $i$. We say that $u$ is \emph{anti-continuous across $\p P$} if
\begin{equation}
\label{AM}
	u_i\big|_{\pO_i \cap \pO_j} = - u_j\big|_{\pO_i \cap \pO_j}
\end{equation}
whenever $\Omega_i$ and $\Omega_j$ are neighbors. We then define the function space
\begin{equation}
\label{H10P}
	H^1_0(M; \p P) = \big\{ u \in H^1(M^o\backslash \p P) : u \text{ vanishes on $\p M$ and 
	is anti-continuous across $\p P$} \big\}.
\end{equation}
Here $M^o = M\backslash \p M$ denotes the interior of $M$, so $u \in H^1(M^o\backslash \p P)$ if and only if $u_i \in H^1(\Omega_i)$ for each $i$. It follows that $H^1_0(M; \p P)$ is complete with respect to the inner product
\begin{equation}
\label{H10Pnorm}
	\left< u, v\right>_{_P} = \sum_{i=1}^k \left<u_i, v_i \right>_{H^1(\Omega_i)}.
\end{equation}
The corresponding norm will be denoted $\| \cdot \|_{_P}$.

We then let $-\DP$ be the self-adjoint operator on $L^2(M)$ generated by the bilinear form
\[
	b(u,v) = \sum_{i=1}^k \int_{\Omega_i} (\nabla u_i \cdot \nabla v_i)\,dV
\]
with $\dom(b) = H^1_0(M; \p P)$.
It is easily shown that $\dom(-\DP)$ consists of all functions $u \in \dom(b)$ that additionally satisfy $\Delta u_i \in L^2(\Omega_i)$ for each $i$ and
\begin{equation}
\label{AM2}
	\frac{\p u_i}{\p \nu_i}\bigg|_{\pO_i \cap \pO_j} = \frac{\p u_j}{\p \nu_j}\bigg|_{\pO_i \cap \pO_j}
\end{equation}
whenever $\Omega_i$ and $\Omega_j$ are neighbors. Note that this is also an anti-continuity condition, since the outward normals satisfy $\nu_i = - \nu_j$ on the common boundary of $\Omega_i$ and $\Omega_j$.

An important observation is that if $P$ is bipartite, then $-\DP$ is unitarily equivalent to the Dirichlet Laplacian on $M$, with the unitary map being multiplication by $1$ on the blue subdomains and $-1$ on the red subdomains. This is a special case of \cite[Proposition~2.2]{BCCM3}; we recall the general statement below in \Cref{prop:hom}.

Since this unitary map preserves nodal sets, in the bipartite case nodal sets of $-\DP$ eigenfunctions coincide with nodal sets of Laplacian eigenfunctions. While $-\DP$ is a more complicated operator than the Dirichlet Laplacian, on account of the internal boundary conditions imposed along $\p P$, it allows us to treat the bipartite and non-bipartite cases simultaneously; cf. \Cref{thm:equality}.

\subsection{Homology of partitions}
\label{ssec:hom}

We finally recall the notion of homology for partition boundary sets that was introduced in \cite{BCCM3}. Suppose $P_1$ and $P_2$ are partitions with corners. If the closure of the symmetric difference of their boundary sets,
\[
	\overline{ \p P_1 \triangle \p P_2} = \overline{(\p P_1 \cup \p P_2) \setminus (\p P_1 \cap \p P_2)},
\]
is the boundary set of a bipartite partition, then we say that $\p P_1$ and $\p P_2$ are \emph{homologous}. In particular,
$\p P_1$ is \emph{null-homologous} (that is, homologous to the empty set) if and only if it is the boundary set of a bipartite partition. This follows from the above definition by choosing $P_2 = \{ M^o\}$, so that $\p P_2 = \varnothing$.

We refer the reader to \cite{BCCM3} for examples and further discussion, including an extension of the definition of homology to less regular boundary sets. An important result from that paper describes how the operator $-\DP$ depends on $P$.

\begin{prop}
\label{prop:hom}
Let $P_1$ and $P_2$ be partitions with corners. If $\p P_1$ is homologous to $\p P_2$, then the operators $-\Delta^{\!\p P_1}$ and $-\Delta^{\!\p P_2}$ are unitarily equivalent.
\end{prop}

A special case of this was noted in the previous section: if $P$ is bipartite, then $\p P$ is null-homologous, therefore $-\DP$ is unitarily equivalent to the Dirichlet Laplacian on $M$.

In general, a small perturbation of $P$ will not have boundary set homologous to $\p P$. It \emph{will} be homologous, however, if the perturbation does not move any of the corner points in the interior on $M$; see \Cref{def:hom}.


\section{Manifolds of diffeomorphisms}
\label{sec:manifold}
We now describe the smooth structure of the relevant spaces of diffeomorphisms. We begin in Section~\ref{ssec:diffeo} with the set of diffeomorphisms of a closed manifold $N$. In Section~\ref{ssec:submanifold} we assume that $N$ is two dimensional and restrict our attention to a subset $M \subset N$ with piecewise smooth boundary, then describe the subset of diffeomorphisms leaving $M$ invariant. We also consider the diffeomorphisms leaving $P$ invariant, where $P$ is a given partition of $M$.

\subsection{Diffeomorphisms of closed manifolds}
\label{ssec:diffeo}
In this section we let $N$ be a closed manifold of arbitrary dimension, and denote by $\Ds$ the set of $H^s$ diffeomorphisms of $N$. We recall the manifold structure of $\Ds$, following the approach of \cite{EM1984} and \cite{Eells}. 

We view $\Ds$ as an open submanifold of the set $H^s(N,N)$ of all $H^s$ maps from $N$ to itself, so we first construct a manifold structure on the latter space, assuming $s > n/2$. Given $f \in H^s(N,N)$, we define
\begin{equation}
\label{tangent}
	T_f H^s(N,N) = \big\{ V \in H^s(N,TN) : V_p \in T_{f(p)}N \text{ for all } p \in N \big\}.
\end{equation}
This is the set of $H^s$ sections of the pullback bundle $f^* TN$. When $f = \id$ it coincides with the space $\mathfrak X^s(N)$ of all $H^s$ vector fields on $N$.

Next, we fix an arbitrary Riemannian metric on $N$ and let $\exp \colon TN \to N$ denote the corresponding exponential map. We get coordinates in a neighbourhood of $f \in H^s(N,N)$ by defining
\begin{equation}
\label{Phi}
	\Phi_f \colon T_f H^s(N,N) \longrightarrow H^s(N,N), \qquad \Phi_f(V) = \exp \circ V.
\end{equation}
That is, $\Phi_f(V)$ is the map $p \mapsto \exp_{f(p)} V_p$. Note that $\Phi_f$ maps the zero section to $f$, since $\exp_{f(p)} (0) = f(p)$ for each $p \in N$. More generally, \eqref{Phi} gives a bijection from a neighbourhood of $0 \in T_f H^s(N,N)$ onto a neighbourhood of $f \in H^s(N,N)$; see \cite{Eells} for details.

It follows that $H^s(N,N)$ is a smooth manifold for any $s > n/2$. Letting $C^1(N,N)$ denote the space of $C^1$ maps from $N$ to itself, the inverse function theorem implies that $\Diff^1(N)$, the set of $C^1$ diffeomorphisms, it an open subset. For $s > n/2 + 1$ the Sobolev embedding theorem says that $H^s(N,N)$ is continuously embedded in $C^1(N,N)$, and so
\[
	\Ds = \big\{ \varphi \in H^s(N,N) : \varphi \in \Diff^1(N) \big\}
\]
is an open subset of $H^s(N,N)$ and hence is a smooth manifold.

\subsection{Boundaries and submanifolds}
\label{ssec:submanifold}
We now restrict the dimension to $n=2$, so $N$ is a compact surface without boundary. Letting $M \subset N$ be a compact, connected manifold with piecewise smooth (or empty) boundary and fixing an $s>2$, we define
\begin{equation}
	\DsM := \big\{ \varphi \in \Ds : \varphi(M) = M \big\}.
\end{equation}
Note that any $\varphi \in \DsM$ maps the boundary of $M$ onto itself, i.e., $\varphi(\p M) = \p M$. We are also interested in diffeomorphisms of $N$ that preserve a given partition with corners, as in \Cref{def:regular}, so we fix such a partition $P$ and define
\begin{equation}
	\DsP = \big\{ \varphi \in \DsM : \varphi(P) = P \big\}.
\end{equation}

We now show that the inclusions
\begin{equation}
	\DsP \subset \DsM  \subset \Ds
\end{equation}
are in fact embeddings.

\begin{theorem}
\label{thm:Ds}
If $M \subset N$ has piecewise smooth boundary and $P$ is partition with corners, then there exists a neighborhood $\cW \subset \Ds$ of the identity such that $\DsM \cap \cW$ is a submanifold of $\Ds \cap \cW$ and $\DsP \cap \cW$ is a submanifold of $\DsM \cap \cW$. \end{theorem}

If $\p M$ and $\Sigma$ are smooth and disjoint, this reduces to \cite[Theorems~6.1 and~6.2]{EM1984}. However, we are interested in the case that $\p M$ is only piecewise smooth and $\Sigma$ may have self-intersections, as well as intersections with $\p M$. The following construction allows us to treat this case; cf. \cite[Lemma~6.4]{EM1984}. 

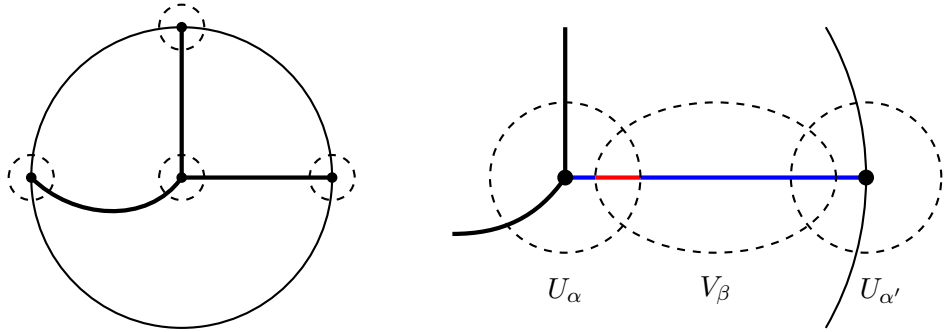
\begin{figure}
\begin{tikzpicture}[scale=1]
	\draw[ultra thick] (2,2) to (2,4);
	\draw[ultra thick] (2,2) to (4,2);
	\draw[ultra thick] (0,2) to [out=315, in=235] (2,2);
	\draw[thick] (4,2) arc (0:360:2);
	\draw[thick,dashed] (2.3,2) arc (0:360:0.3);
	\draw[thick,dashed] (2.3,4) arc (0:360:0.3);
	\draw[thick,dashed] (0.3,2) arc (0:360:0.3);
	\draw[thick,dashed] (4.3,2) arc (0:360:0.3);
	\fill (0,2) circle[radius=2pt];
	\fill (2,2) circle[radius=2pt];
	\fill (4,2) circle[radius=2pt];
	\fill (2,4) circle[radius=2pt];
\end{tikzpicture}
\hspace{1cm}
\begin{tikzpicture}[scale=1]
	\draw[ultra thick] (0,0) to (0,2);
	\draw[ultra thick, blue] (0,0) to (0.4,0);
	\draw[ultra thick, red] (0.4,0) to (1,0);
	\draw[ultra thick, blue] (1,0) to (4,0);
	\draw[ultra thick] (-1.5,-0.75) to [out=0, in=235] (0,0);
	\draw[thick,dashed] (2,0) ellipse (1.6 and 1);
	\draw[thick,dashed] (1,0) arc (0:360:1);
	\draw[thick,dashed] (5,0) arc (0:360:1);
	\draw[thick] (4,0) arc (0:30:4);
	\draw[thick] (4,0) arc (360:330:4);
	\node at (0,-1.5) {$U_{\alpha}$};
	\node at (2,-1.5) {$V_{\beta}$};
	\node at (4.2,-1.5) {$U_{\alpha'}$};
	\fill (0,0) circle[radius=3pt];
	\fill (4,0) circle[radius=3pt];
\end{tikzpicture}
\caption{Left: a 3-partition of the disk with a ball around each corner point. Right: a magnification showing a covering of part of $\p M \cup\Sigma$, as in the proof of Lemma~\ref{lem:geo}. The curve $\gamma$ (arising in step (3)) is shown in red, and the union of the blue and red curves is the set $\Gamma$, the smooth component of $\p M \cup \Sigma$ that contains $\gamma$.}
\label{fig:nbhd}
\end{figure}

\begin{lemma}
\label{lem:geo}
There exists a smooth Riemannian metric $\bar g$ on $N$ such that the smooth part of $\p M \cup \Sigma$ is totally geodesic.
\end{lemma}

We use the notation $\bar g$ to clarify that this is different from the metric $g$ whose Laplace--Beltrami operator we are studying.

\begin{proof}
It suffices to construct $\bar g$ in a neighborhood of $\p M \cup \Sigma$. We do this in three steps.

1) In a neighborhood of each corner point of $\p M \cup \Sigma$ we can find coordinates in which $\p M \cup \Sigma$ is given by a finite number of line segments meeting at the origin. Since the corner points are finite in number, we can assume that these neighborhoods, which we denote $\{U_\alpha\}$, are mutually disjoint; see \Cref{fig:nbhd}(left). In $U_\alpha$ we use these coordinates to define a Riemannian metric $g^\alpha$ by $g^\alpha_{ij} = \delta_{ij}$, so that straight lines in these coordinates are geodesics and hence the smooth part of $(\p M \cup \Sigma) \cap U_\alpha$ is totally geodesic with respect to $g^\alpha$.

2) The smooth part of $\p M \cup \Sigma$ consists of disjoint open curves, so we can find open sets $\{V_\beta\}$ that are mutually disjoint, do not contains any corner points and, together with the $\{U_\alpha\}$, cover all of $\p M \cup \Sigma$; see \Cref{fig:nbhd}(right). Each $(\p M \cup \Sigma) \cap V_\beta$ consists of a smooth curve, so we can easily construct a metric $g^\beta$ on $V_\beta$ for which this is totally geodesic.

3) We have thus covered $\p M \cup \Sigma$ with finitely many open sets, $\{U_\alpha\}$ and $\{V_\beta\}$, such that the $U_\alpha$ are mutually disjoint, as are the $V_\beta$, and each $U_\alpha \cap V_\beta$ is either empty or consists of a smooth curve. On each of these sets we have a Riemannian metric with respect to which the smooth part of $\pO \cup \Sigma$ is totally geodesic. To complete the proof we need to patch together these locally defined metrics in a way that preserves the totally geodesic condition. This will be done using Fermi coordinates.

First, we use a partition of unity to obtain a metric $\tilde g$ on $N$ that, for any $\alpha$ and $\beta$, agrees with $g^\alpha$ on $U_\alpha \backslash \overline V_\beta$ and agrees with $g^\beta$ on $V_\beta \backslash \overline U_\alpha$. There is no reason for the part of $\p M \cup \Sigma$ in the overlap $\overline{U_\alpha \cap V_\beta}$ to be totally geodesic in this new metric, so we need to modify $\tilde g$ there.

For fixed $\alpha$ and $\beta$, $(\p M \cup \Sigma) \cap (U_\alpha \cap V_\beta)$ is a smooth curve. We denote this curve by $\gamma$, and let $\Gamma$ be the component of the smooth part of $\p M \cup \Sigma$ that contains $\gamma$; see \Cref{fig:nbhd}(right). Let $(x,y)$ be Fermi coordinates, with respect to $\tilde g$, in a neighborhood of $\Gamma$. That is, $x$ is an arbitrary coordinate along $\Gamma$ and $y$ is the (signed) normal distance from $\Gamma$. We can assume that $\gamma$ corresponds to $(-1,1) \times \{0\}$ in these coordinates. The metric has the form
\[
	\tilde g = \tilde a(x,y) dx^2 + dy^2,
\]
where $\tilde a$ is a smooth, positive function. The second fundamental form of $\Gamma$ is proportional to $\tfrac{\p \tilde a}{\p y}(x,0)$. In particular, we have $\tfrac{\p \tilde a}{\p y}(x,0) = 0$ for $|x| \geq 1$ by construction. We then define
\begin{equation}
\label{adef}
	a(x,y) = \tilde a(x,y) - y\, \chi(y) \frac{\p \tilde a}{\p y}(x,0),
\end{equation}
where $\chi(y)$ is a smooth cutoff function that equals 1 in a neighborhood of $y=0$ and vanishes for $|y| \geq \delta$. This implies
\[
	\left| y\, \chi(y) \frac{\p \tilde a}{\p y}(x,0)\right| \leq \delta \left|\frac{\p \tilde a}{\p y}(x,0)\right|.
\]
Since $\tfrac{\p \tilde a}{\p y}(x,0)$ is bounded (it is supported in the compact set $[-1,1]$) and $\tilde a$ is bounded away from zero, we can choose $\delta$ small enough that the function $a(x,y)$ defined by \eqref{adef} is positive, and so
\[
	\bar g = a(x,y) dx^2 + dy^2
\]
defines a Riemannian metric on $U_\alpha \cup V_\beta$. It follows from \eqref{adef} that $\tfrac{\p a}{\p y}(x,0) = 0$ for all $x$, and so $\Gamma$ is totally geodesic with respect to $\bar g$.

Similarly modifying $\tilde g$ on all other $U_\alpha \cap V_\beta$, we obtain the desired metric $\bar g$.
\end{proof}

We let $\mathfrak{X}^s_M(N) \subset \mathfrak{X}^s(N)$ denote the closed subspace of vector fields that are tangent to the smooth part of $\p M$. Note that such a vector field must vanish at every corner point of $\p M$. We also let $\mathfrak{X}^s_{P}(N) \subset \mathfrak{X}^s_M(N)$ denote the closed subspace of vector fields that are tangent to the smooth part of $\p M \cup \Sigma$, and observe that these vector fields vanish at every corner point of $\p M \cup \Sigma$.

\begin{lemma}
\label{lem:epsilon}
With $\bar g$ as in Lemma~\ref{lem:geo} and $s > 2$, there exists $\epsilon > 0$ such that if $V \in \mathfrak{X}^s(N)$ and $\|V\|_{s} < \epsilon$, then $\varphi = \exp \circ V$ satisfies
\begin{enumerate}
	\item $\varphi \in \DsM$ if and only if $V \in \mathfrak{X}_M^s(N)$,
	\item $\varphi  \in \DsP$ if and only if $V \in \mathfrak{X}_P^s(N)$.
\end{enumerate}
\end{lemma}

In fact it suffices to have the $C^1$ norm of $V$ sufficiently small. The requirement that $s>2$ gives the embedding $H^s \subset C^1$ since $n=2$.

\begin{proof}
We only prove the first assertion, as the second one follows from the same argument.

We first take $\epsilon > 0$ small enough that $\|V\|_s < \epsilon$ implies $\Phi(V) = \exp \circ V$ is a diffeomorphism. Next, we claim that there exists $\epsilon > 0$ such that $V \in \mathfrak{X}_M^s(N)$ and $\|V\|_s < \epsilon$ imply $\varphi(\p M) \subset \p M$. Since corner points of $\p M$ are fixed by $\varphi$, we just need to consider points in the smooth part of the boundary, which we denote $\p M_*$ for convenience.

Given $x \in \p M_*$, let $\gamma(t)$ denote the unique $\bar g$-geodesic with $\gamma(0) = x$ and $\gamma'(0) = V(x)$, so that $\varphi(x) = \exp_x V(x) = \gamma(1)$. The fact that $\p M_*$ is totally geodesic and $V(x)$ is tangent to $\p M_*$ guarantees $\gamma(t) \in \p M_*$ for small $t$, but not necessarily for large $t$, since $\gamma(t)$ will leave $\p M$ if it passes through a corner point. This will happen if $|V(x)|$ is larger than the distance (in $\p M$) from $x$ to the nearest corner point, so we must ensure $V$ is small enough that this cannot happen.
Letting $p$ denote the nearest corner point to $x$, Taylor's theorem gives
\[
	|V(x)| \leq |\nabla V(x')| \,d(p,x)
\]
for some point $x'$ between $p$ and $x$, since $V(p) = 0$. On the other hand, by the Sobolev embedding theorem there is a constant $K>0$ so that $|\nabla V(x')| \leq K \|V\|_s$ for any $x' \in M$. We thus choose $\epsilon < 1/K$, so that $\|V\|_s < \epsilon$ implies $|V(x)| < d(x,p)$, as required.

Therefore, if $V \in \mathfrak{X}_M^s(N)$ and $\|V\|_s < \epsilon$, we have that $\varphi$ is a diffeomorphism and $\varphi(\p M) \subset \p M$. We claim that $\varphi(\p M) = \p M$. For $\epsilon$ sufficiently small $\varphi$ maps each connected component of $\p M$ into itself. If $\p M$ has a smooth component (which must be diffeomorphic to $S^1$), it is easily shown that the image of this component is both open and closed, and hence it is mapped \emph{onto} itself. 
On the other hand, the image of a smooth segment of $\p M$ will be a connected subset of this segment that contains both endpoints, since they are fixed by $\varphi$, and so $\varphi$ maps this segment onto itself. It follows that $\varphi(\p M) = \p M$.

Since $\varphi$ is a diffeomorphism, it therefore maps $N \backslash \p M$ onto itself. In particular, each connected component of $N \backslash \p M$ is mapped to a connected component of $N \backslash \p M$. We thus choose $\epsilon$ small enough that $\varphi$ maps each connected component onto itself, and hence for $\|V\|_s < \epsilon$ we get $\varphi(M) = M$, as was to be shown. This proves the first implication.

To prove the converse, let $V \in \mathfrak{X}^s(N)$, with $\|V\|_s < \epsilon$, and suppose that $\varphi \in \DsM$. Since $\varphi$ is a diffeomorphism, it follows that $\varphi(\p M) = \p M$. Moreover, $\varphi$ fixes all of the corner points and maps each connected component of $\p M_*$ onto itself. In particular, this implies that $V$ vanishes at all corner points.

For $x \in \p M$, we have $\varphi(x) = \gamma(1)$, where $\gamma(t)$ is the geodesic with $\gamma(0) = x$ and $\gamma'(0) = V(x)$. Since $N$ is compact, we can choose $\epsilon>0$ small enough that $\|V\|_s < \epsilon$ implies $x$ and $\varphi(x)$ are contained in a geodesically convex neighborhood $U_x$, for each $x \in N$. This means there is a unique minimizing geodesic from $x$ to $\varphi(x)$. On the other hand, we can choose $\epsilon$ so that $\p M_* \cap U_x$ is connected and has length less than the injectivity radius of $(N,g)$. This means the part of $\p M_*$ between $x$ and $\varphi(x)$ is minimizing, and hence is the unique minimizing geodesic between these two points, i.e., it is the curve $\gamma(t)$ introduced above. It follows that $\gamma(t) \in \p M_*$ for all $0 \leq t \leq 1$, therefore $\gamma'(0) = V(x)$ is tangent to $\p M_*$, as was to be shown.
\end{proof}

We are now ready to prove our theorem about submanifolds of $\Ds$.

\begin{proof}[Proof of Theorem~\ref{thm:Ds}]
Choose open neighborhoods $\cU \subset \mathfrak{X}^s_M(N)$ and $\cV \subset \mathfrak{X}^s_M(N)^\perp$, both containing $0$, so that $\cU \times \cV$ is contained in
the $\epsilon$ ball around $0$ in $\mathfrak{X}^s(N)$, with $\epsilon$ as in Lemma~\ref{lem:epsilon}. Now consider
\begin{equation}
	\Phi\big|_{\cU \times \cV} \colon \cU \times \cV \longrightarrow \Ds
\end{equation}
with $\Phi = \Phi_{\id}$ as in \eqref{Phi}. This is a bijection from $\cU \times \cV$ onto $\cW = \Phi(\cU \times \cV)$, and it follows immediately from Lemma~\ref{lem:epsilon} that $\Phi(U+V) \in \cD^s_M(N)$ if and only if  $V=0$. We have thus found a coordinate chart $\cU \times \cV$ for $\cW$ such that $\DsM \cap \cW$ corresponds to $\cU \times \{0\}$, so we conclude that $\DsM \cap \cW$ is a submanifold of $\Ds \cap \cW$. The same argument, combined with the second half of Lemma~\ref{lem:epsilon}, shows that $\DsP \cap \cW$ is a submanifold of $\DsM \cap \cW$.
\end{proof}

\begin{rem}
It is convenient to work with Hilbert spaces since it is guaranteed that every closed subspace has a closed complement. If $\p M \cup \Sigma$ is smooth, it can be directly shown that the space of $C^k$ or $C^{k,\alpha}$ vector fields that are tangent to $\p M \cup \Sigma$ has a closed complement; see \cite[Remark~12.2(i)]{EM1984}. This is more delicate, however, when intersections are present.
\end{rem}


\section{Variational formulas}
\label{sec:variation}

In this section, we compute the first and second variation of a simple Dirichlet eigenvalue with respect to a deformation of the domain. While we only consider $C^2$ deformations, we are interested in non-smooth domains, in particular those with corners, so we use the method of interior variations, which applies to open, bounded sets in a Riemannian manifold of arbitrary dimension.

The main result of this section is \Cref{interior}. With an additional regularity assumption, we recover the classical formula of Hadamard for the first derivative, written as a boundary integral; see \Cref{thm:Hadamard}. This will be used to characterize critical points of the energy functional $\lap$ in \eqref{def:lambda}. The interior formula \eqref{Had2} for the second derivative is sufficient to compute the Hessian of $\lap$, but in \Cref{app:HadamardC3} we show how it can be used to obtain the classical formula for the second derivative on a $C^3$ domain.

For a discussion of Hadamard's formula with minimal regularity assumptions (on both the domain and the deformation) we refer the reader to \cite{Koz06,Koz12,Koz20}.

\subsection{Interior variations}
\label{sec:interior}

Let $\Diff^k(N)$ denote the set of $C^k$ diffeomorphisms on $N$. For an open set $\Omega \subset N$ we let $W^{1,2}(\Omega)$ be the set of $L^2$ functions with distributional gradient in $L^2(\Omega)$. Following \cite{M00}, we have that $H^1(\Omega) = \{ U\big|_\Omega : U \in W^{1,2}(N)\}$ is a subset of $W^{1,2}(\Omega)$. If $\Omega$ has Lipschitz boundary these spaces coincide, but in general the inclusion can be proper.

\begin{rem}
For our study of spectral mininal partitions, we only consider the Dirichlet problem on piecewise smooth domains. However, in \Cref{interior}, we allow $\Omega$ to be an arbitrary open set and consider more general boundary conditions, as we expect this result to be useful elsewhere, for instance in the study of Neumann domains \cite{BCE}. 
\end{rem}

For any closed subspace $\cV \subset W^{1,2}(\Omega)$ that contains $W^{1,2}_0(\Omega)$, we let $-\Delta^\cV_\Omega$ denote the self-adjoint operator on $L^2(\Omega)$ generated by the closed bilinear form $(u,v) \mapsto \int_\Omega \left<\nabla u, \nabla v\right>dV$ on $\cV$. The Dirichlet and Neumann Laplacians correspond to $\cV = W^{1,2}_0(\Omega)$ and $\cV = W^{1,2}(\Omega)$, respectively. Given a diffeomorphism $\varphi \in \Diff^1(N)$, we define $\varphi(\cV) = \{ u \circ \varphi^{-1} : u \in \cV\} \subset W^{1,2}\big(\varphi(\Omega)\big)$ and let $-\Delta^{\varphi(\cV)}$ denote the corresponding operator on $L^2\big(\varphi(\Omega)\big)$.

\begin{lemma}
\label{lem:varIFT}
Let $(N,g)$ be a closed Riemannian manifold, $\Omega \subset N$ an open, bounded set, and $\lambda$ a simple eigenvalue of $-\Delta^\cV$ with normalized eigenfunction $\psi$. There exists an open neighborhood $\cU \subset \Diff^1(N)$ containing the identity and a smooth function $\cU \ni \varphi \mapsto (\psi_\varphi, \lambda_\varphi) \in \cV \times \bbR$ such that $\psi_{\id} = \psi$, $\lambda_{\id} = \lambda$, and $\lambda_\varphi$ is an eigenvalue of $-\Delta^{\varphi(\cV)}$ with normalized eigenfunction $\psi_\varphi \circ \varphi^{-1}$.
\end{lemma}

We omit the proof, which is a standard application of the implicit function theorem; see, for instance,  \cite[\textsection5.7]{Henrot}. We now give explicit formulas for the first and second derivative of $\lambda_{\varphi_t}$, where $\varphi_t$ is a one-parameter family of diffeomorphisms.

\begin{prop}
\label{interior}
Assume the hypotheses of \Cref{lem:varIFT} and let $\varphi_t$ be a family of diffeomorphisms with $\varphi_0 = \id$, so that $\varphi_t'(x) = X_t(\varphi_t(x))$ for some vector field $X_t$. Writing $\lambda_t = \lambda_{\varphi_t}$, we let $\lambda'$ denote $\lambda_t'$ evaluated at $t=0$, and similarly for $X$, $X'$ and $\lambda''$.

\begin{enumerate}
	\item If $\varphi_t \in C^1\big( (-\epsilon,\epsilon), \Diff^1(N)\big)$, then
\begin{equation}
\label{Had1}
	\lambda' = \int_\Omega \dv Y \,dV,
\end{equation}
where
\begin{equation}\label{Ydef}
	Y = \big( |\nabla \psi|^2 - \lambda \psi^2 \big) X - 2 (\nabla_X\psi) \nabla\psi.
\end{equation}

\item If $\varphi_t \in C^2\big( (-\epsilon,\epsilon), \Diff^1(N)\big)$, $X \in C^2(N)$ and $\psi \in W^{2,2}(\Omega)$, then
\begin{equation}
\label{Had2}
	\lambda'' = \int_\Omega \dv\big[ X(\dv Y) +  \tilde Y + 2Z - \lambda' \psi^2 X \big] \,dV,
\end{equation}
with $Y$ defined in \eqref{Ydef} and
\begin{align}
	\tilde Y &= \big( |\nabla \psi|^2 - \lambda \psi^2 \big) X' - 2 (\nabla_{X'}\psi) \nabla\psi, \label{Ytildedef}\\
	Z &= \big(\! \left<\nabla\psi, \nabla w\right> - \lambda \psi w \big) X - (\nabla_X\psi) \nabla w - (\nabla_X w)\nabla\psi, \label{Zdef}
\end{align}
where $w \in W^{1,2}(\Omega)$ is the unique solution to
\begin{equation}
\label{w:pde}
	\Delta w + \lambda w + \lambda'\psi = 0, \qquad w\big|_{\pO} = - (X \cdot\nu) \frac{\p\psi}{\p\nu}, \qquad \int_\Omega w\psi\,dV = 0.
\end{equation}
\end{enumerate}
\end{prop}

\begin{rem}
\label{rem:reg1}
Note that \eqref{Had1} only requires $\psi \in W^{1,2}(\Omega)$. While $Y$ depends on $\nabla \psi$, its divergence does not contain second derivatives, due to a cancellation of Hessian terms and the fact that $-\Delta\psi =\lambda\psi$; see \eqref{eq:Ydiv}. (The cancellation of Hessians is valid for the weak divergence, since $\psi \in C^\infty(\Omega)$ by interior regularity.) The same is true of the integrand in \eqref{Had2}, and in fact this result is valid for any $\psi \in W^{1,2}(\Omega)$; see \Cref{rem:reg2}. We have imposed the extra assumption $\psi \in W^{2,2}(\Omega)$ here, as it greatly simplifies the proof and is sufficient for the purposes of this paper.
\end{rem}

Before giving the proof (in \Cref{sec:intproof}) we describe some consequences of \eqref{Had1} and \eqref{Had2} for domains with sufficiently regular boundary. These will be used  in the proofs of \Cref{thm:critical,thm:Hess1}.

\subsection{Hadamard's formula}
From \Cref{interior} we obtain a version of Hadamard's formula for the Dirichlet problem.

\begin{theorem}
\label{thm:Hadamard}
If $\Omega$ has Lipschitz boundary and $\psi \in H^2(\Omega) \cap H^1_0(\Omega)$, then
\begin{equation}
\label{smoothHad1}
	\lambda' = -\int_{\pO}\left(\frac{\p\psi}{\p\nu}\right)^2 (X\cdot\nu) \, d\mu.
\end{equation}
\end{theorem}

In particular, \eqref{smoothHad1} holds in any dimension if $\pO$ is $C^2$, and in two dimensions if $\pO$ is piecewise $C^2$ with convex corners, by \cite[Remark~3.2.4.6]{Gr}.

Before starting the proof, we recall that the divergence theorem
\begin{equation}
\label{div:weak}
	\int_\Omega \dv Y \,dV = \int_{\pO} (Y \cdot \nu) \,d\mu 
\end{equation}
holds for any vector field $Y \in W^{1,1}(\Omega)$ on a Lipschitz domain $\Omega$. For such $Y$ we have $Y \cdot \nu \in L^1(\pO)$, i.e., the normal component of $Y$ is a function, rather than a distribution, as would be the case if $Y$ were less regular; cf. \cite[Theorem~2.2]{Girault}.

\begin{proof}[Proof of \Cref{thm:Hadamard}]
The assumption $\psi \in H^2(\Omega)$ implies the vector field $Y$ defined in \eqref{Ydef} is in $W^{1,1}(\Omega)$, therefore
\[
	\lambda' = \int_{\pO} \dv Y\,dV = \int_{\pO} (Y\cdot\nu) \,d\mu.
\]
Since $\psi$ vanishes on $\pO$, its gradient is purely normal, i.e., $\nabla\psi\big|_{\pO} = \frac{\p\psi}{\p\nu} \nu$. This implies
\[
	Y\cdot\nu = \left(\frac{\p\psi}{\p\nu}\right)^2 (X\cdot\nu) - 2 \left(\frac{\p\psi}{\p\nu}\right)^2 (X\cdot\nu) = - \left(\frac{\p\psi}{\p\nu}\right)^2 (X\cdot\nu)
\]
and \eqref{smoothHad1} follows.
\end{proof}

\subsection{Proof of \Cref{interior}}
\label{sec:intproof}

In the proof we will refer to the ``inverse" metric tensor $g^\sharp$. This is a symmetric $(2,0)$ tensor given by $g^\sharp(\alpha_1,\alpha_2) = g(\alpha_1^\sharp, \alpha_2^\sharp)$ on one-forms $\alpha_1,\alpha_2$, where $\alpha_i^\sharp$ denotes the vector field dual to $\alpha_i$. We then write $\varphi_t^* g^\sharp$ for the pushforward by $\varphi_t^{-1}$, i.e. $\big(\varphi_t^* g^\sharp\big)(\alpha,\beta) = g^\sharp\big( (\varphi_t^{-1})^*\alpha, (\varphi_t^{-1})^*\beta)$. Finally, we denote by $g_t = (\varphi_t^* g^\sharp)^\flat$ the corresponding $(0,2)$ tensor. This is a one-parameter family of Riemannian metrics, which in general is not the same as $\varphi_t^* g$.

We will use \cite[Corollaire~5.2.5]{Henrot}, which says how to differentiate functions on a moving domain.

\begin{lemma}
\label{lem:Henrot}
Suppose $\varphi_t \colon (-\epsilon,\epsilon) \to \Diff^1(N)$ is differentiable at $0$, with $\varphi_0 = \id$ and $\varphi_t'\big|_{t=0} = X$. Define $\Omega_t = \varphi_t(\Omega)$ and suppose $f_t \in L^1(\Omega_t)$ for each $t$. If $f_0 \in W^{1,1}(\Omega)$ and the function $F_t := f_t \circ \varphi_t \in L^1(\Omega)$ is differentiable at $t=0$, then $I(t) = \int_{\Omega_t} f_t\,dV$ is differentiable at $t=0$, with
\begin{equation}
\label{Ideriv}
	I'(0) = \int_\Omega \Big\{ f'_0 + \dv(f_0 X)\Big\}\,dV,
\end{equation}
where the function $f_0'$ satisfies
\begin{equation}
\label{f0'}
	f_0'\big|_K = \frac{d}{dt} \left(f_t\big|_K\right) \bigg|_{t=0}
\end{equation}
for any compact $K \subset \Omega$.
\end{lemma}

The hypothesis $f_0 \in W^{1,1}(\Omega)$ is the reason for our assumption $\psi \in W^{2,2}(\Omega)$ in \Cref{interior}.

\begin{rem}
The equation \eqref{Ideriv} is easily obtained by a formal calculation, but one has to be careful when differentiating $f_t$, since the domains are changing. The content of \Cref{lem:Henrot} is that it suffices to differentiate $f_t\big|_K$ for compact $K$. (This is well defined because $K \subset \Omega$ implies $K \subset \Omega_t$ for small enough $t$.) That is, there exists a function $f_0'$ satisfying \eqref{f0'} for each $K$ such that \eqref{Ideriv} holds.
\end{rem}

\begin{proof}[Proof of \Cref{interior}]
For convenience we set $\psi_t = \psi_{\varphi_t}$ and $\lambda_t = \lambda_{\varphi_t}$. Since $\psi_t \circ \varphi_t^{-1}$ is $L^2(\Omega_t)$-normalized, we have
\begin{equation}
\label{eval1}
	\lambda_t = \int_{\Omega_t} g\big(\nabla(\psi_t \circ \varphi_t^{-1}), \nabla(\psi_t \circ \varphi_t^{-1}) \big) \,dV.
\end{equation}
By the chain rule, the integrand on the right-hand side is
\begin{align*}
	g\big(\nabla(\psi_t \circ \varphi_t^{-1}), \nabla(\psi_t \circ \varphi_t^{-1}) \big) &= g^\sharp\big(d(\psi_t \circ \varphi_t^{-1}), d(\psi_t \circ \varphi_t^{-1}) \big) \\
	&= g^\sharp\big( (\varphi_t^{-1})^* d\psi_t, (\varphi_t^{-1})^* d\psi_t \big) \\
	&= \big(\varphi_t^* g^\sharp \big)(d\psi_t, d\psi_t),
\end{align*}
so we can rewrite \eqref{eval1} as
\begin{equation}\label{eq:Lbi}
	\lambda_t 
	= \int_\Omega g_t(\nabla \psi_t, \nabla \psi_t) \, \varphi_t^*dV,
\end{equation}
with $g_t = (\varphi_t^* g^\sharp)^\flat$ as above.
Differentiating and using the identity\footnote{This is not just the definition of the Lie derivative, since $\varphi_t$ need not be a one-parameter \emph{group}; see \cite[Appendix~B]{Geiges}.}
\[
	\frac{\p}{\p t} \varphi_t^*dV = \varphi_t^* (L_{X_t} dV) = \varphi_t^*( \dv X_t\,dV), 
\]
we get
\begin{align}
\begin{split}
\label{eq:LPbi}
	\lambda'  = \int_\Omega \Big\{ g'(\nabla\psi, \nabla\psi) +  2g(\nabla\psi, \nabla\psi')
	+ g( \nabla\psi, \nabla\psi) \dv X& \Big\} \,dV.
\end{split}
\end{align}
To eliminate the $\psi'$ term we observe that
\[
	2 \int_\Omega g (\nabla \psi, \psi')\,dV = 2 \lambda \int_\Omega \psi \psi' \,dV
	= - \lambda \int_\Omega \psi^2 \dv X\,dV,
\]
where the first equality holds because $\psi$ is an eigenfunction of  $-\Delta^{\cV}$ and $\psi' \in \cV$, and the second equality comes from differentiating the normalization condition $\int_{\Omega} \psi_t^2\,\varphi_t^* dV = 1$. Substituting this in \eqref{eq:LPbi} and using $g' =  - \cL_X g$ gives
\begin{equation}
\label{L'intermediate}
	\lambda' = \int_\Omega \Big\{ \big(|\nabla\psi|^2 - \lambda \psi^2\big)\dv X - 2g\big(\nabla_{\nabla\psi}X, \nabla\psi\big) \Big\}\,dV.
\end{equation}
To complete the proof of \eqref{Had1} we recall the definition of $Y$ in \eqref{Ydef} and compute
\begin{align}
\begin{split}
\label{eq:Ydiv}
	\dv Y &= 2 \Hess \psi(\nabla\psi, X) - 2\lambda \psi \nabla_X\psi + \big(|\nabla\psi|^2 - \lambda\psi^2\big) \dv X \\
	&\quad - 2 \Hess \psi(X, \nabla\psi) - 2 g\big(\nabla_{\nabla\psi} X, \nabla\psi\big) - 2 (\nabla_X\psi) \Delta \psi.
\end{split}
\end{align}
Cancelling the Hessians and using $\Delta\psi + \lambda \psi = 0$, we see that this equals the integrand in \eqref{L'intermediate}.

For the second variation, we observe that \eqref{Had1} is valid for any open, bounded domain, so we can apply it on $\Omega_t$ to get
\begin{equation}
\label{L'}
	\lambda_t' = \int_{\Omega_t} \dv Y_t\,dV,
\end{equation}
where $Y_t = \big( |\nabla u_t |^2 - \lambda_t u_t^2 \big) X_t - 2 (\nabla_Xu_t ) \nabla u_t$ and we have set $u_t = \psi_t \circ \varphi_t^{-1}$ for convenience. We will use \Cref{lem:Henrot} to differentiate \eqref{L'}.

We first compute
\begin{equation}
\label{divYt}
	\dv Y_t = \big( |\nabla u_t |^2 - \lambda_t \psi^2 \big) \dv X_t - 2g(\nabla_{\nabla u_t} X_t, \nabla u_t).
\end{equation}
The assumption $\varphi_t \in C^2\big( (-\epsilon,\epsilon), \Diff^1(N)\big)$ implies $X_t \in C^1\big( (-\epsilon,\epsilon), C^1(N)\big)$. From this and the fact that $\psi_t \in C^1\big((-\epsilon,\epsilon), W^{1,2}(\Omega) \big)$, we see that $(\dv Y_t) \circ \varphi_t$ is differentiable in $L^1(\Omega)$. Moreover, the assumptions $X \in C^2(N)$ and $\psi \in W^{2,2}(\Omega)$ together guarantee $\dv Y \in W^{1,1}(\Omega)$, so the hypotheses of \Cref{lem:Henrot} are satisfied and we conclude that
\[
	\lambda'' = \int_\Omega \Big\{ (\dv Y_t)'\big|_{t=0} + \dv\big[X(\dv Y)\big] \Big\}\,dV.
\]
We claim that
\begin{equation}
\label{claimYt}
	(\dv Y_t)'\big|_{t=0} = \dv \big[\tilde Y + 2Z - \lambda' \psi^2 X\big],
\end{equation}
from which \eqref{Had2} immediately follows.

To prove \eqref{claimYt}, we first appeal to \cite[Th\'eor\`eme~5.7.4]{Henrot}, which says $u'$ is the unique solution to \eqref{w:pde}, and hence $u' = w$.  Differentiating \eqref{divYt} at $t=0$, we obtain
\begin{align*}
	(\dv Y_t)'\big|_{t=0} &= \big( 2g(\nabla\psi, \nabla w) - 2\lambda \psi w - \lambda' \psi^2 \big) \dv X \\
	&\quad + \big(|\nabla\psi|^2 - \lambda \psi^2\big) \dv X' - 2 g(\nabla_{\nabla\psi} X', \nabla\psi) \\
	&\quad - 2g(\nabla_{\nabla\psi} X, \nabla w) - 2g(\nabla_{\nabla w} X, \nabla \psi).
\end{align*}
To see that this equals $\dv \big[ \tilde Y + 2Z - \lambda' \psi^2 X\big]$, we compute
\begin{align*}
	\dv\big[\psi^2 X\big] &= 2 \psi \nabla_X \psi + \psi^2 \dv X, \\
	\dv \tilde Y &= \big( |\nabla \psi|^2 - \lambda \psi^2 \big) \dv X' - 2g(\nabla_{\nabla\psi} X', \nabla\psi),
\end{align*}
and
\[
	\dv Z =  \big(g(\nabla\psi, \nabla w) - \lambda \psi w \big) \dv X - g(\nabla_{\nabla\psi} X, \nabla w) - g(\nabla \psi, \nabla_{\nabla w} X) + \lambda' \psi \nabla_X\psi,
\]
which completes the proof.
\end{proof}

\begin{rem}
\label{rem:reg2}
It is possible to write $(\dv Y_t) \circ \varphi_t$ explicitly in terms of the metric $g_t$ and differentiate $\lambda_t' = \int_\Omega [(\dv Y_t) \circ \varphi_t] \,\varphi_t^*dV$ directly, as was done for the first variation. After a lengthy computation this yields \eqref{Had2} without assuming $\psi \in W^{2,2}(\Omega)$. The assumption $\psi \in W^{2,2}(\Omega)$ allows us to use \Cref{lem:Henrot} and as a result simplifies the proof considerably; see \Cref{rem:reg1}.
\end{rem}

\subsection{The second variation of a simple eigenvalue on a $C^3$ domain}
\label{app:HadamardC3}
Assuming more regularity of $\pO$, we get the following analogue of Hadamard's formula for the second variation.

\begin{theorem}
\label{thm:2var}
If $\Omega$ has $C^3$ boundary, then
\[	\lambda'' 
	 = \int_{\pO} \bigg\{2 w \frac{\p w}{\p\nu} + \left(\frac{\p\psi}{\p\nu}\right)^2 \!\Big(- (X' \cdot\nu) - g(\nabla_\nu X, \nu)(X \cdot\nu) + \nabla_T(X\cdot\nu) + H(X\cdot\nu)^2  \Big)\bigg\} \, d\mu,
\]
where $T := X - (X\cdot\nu)\nu$ is the tangential part of $X$.
\end{theorem}

This result is well known; see, for instance, \cite[eq.~(150)]{G10} or \cite[p.~226]{Henrot}. While we do not use it in this paper (since we are interested in domains with corners), some of the intermediate calculations will be useful in \Cref{ssec:pfHess1}, when we prove \Cref{thm:Hess1}. Moreover, the derivation of $\lambda''$ via the method of interior variations may be of independent interest.

\begin{rem}
The deformed domain $\varphi_t(\Omega)$ depends on both the normal and tangential parts of $X$, but the tangential contribution only shows up at second order; cf. \eqref{smoothHad1} where it is absent. In \cite{G10}, the deformation is chosen to be purely normal, so $T$ does not appear.
\end{rem}

\begin{proof}[Proof of \Cref{thm:2var}]
Since $\Omega$ has $C^3$ boundary, we have $\psi \in H^3(\Omega)$. It follows from \eqref{w:pde} that $w\big|_{\pO} \in H^{3/2}(\pO)$ and hence $w \in H^2(\Omega)$. This means the vector field $X(\dv Y) +  \tilde Y + 2Z - \lambda' \psi^2 X$ is in $W^{1,1}(\Omega)$. Denoting this by $W$, we can apply the divergence theorem to \eqref{Had2} to obtain $\lambda'' = \int_{\pO} (W\cdot\nu)\,d\mu$.

Using the fact that $\psi$ vanishes on the boundary, as in the proof of \Cref{thm:Hadamard}, we find
\begin{align*}
	\dv Y\big|_{\pO} &= \left(\frac{\p\psi}{\p\nu}\right)^2 \big( \dv X - 2 g(\nabla_\nu X, \nu)\big), \\
	\tilde Y \cdot\nu &= - \left(\frac{\p\psi}{\p\nu}\right)^2 (X' \cdot\nu) \\
	Z \cdot \nu &= -\frac{\p\psi}{\p\nu} \nabla_X w = w\frac{\p w}{\p\nu} + \frac{\p\psi}{\p\nu} \nabla_T \left( (X\cdot\nu) \frac{\p\psi}{\p\nu}\right) 
\end{align*}
where for the $Z$ term we have decomposed $\nabla _X w = \nabla_T w + (X\cdot\nu) \frac{\p w}{\p\nu}$ and used the fact that $w\big|_{\pO} = - (X\cdot\nu) \frac{\p\psi}{\p\nu}$. We thus have the boundary term
\begin{equation}\label{bdry1}
	W\cdot\nu = 2 w\frac{\p w}{\p\nu} + 2 \frac{\p\psi}{\p\nu} \nabla_T \left( (X\cdot\nu) \frac{\p\psi}{\p\nu}\right) + \left(\frac{\p\psi}{\p\nu}\right)^2 \Big(- (X' \cdot\nu) + \big(\dv X - 2 g(\nabla_\nu X, \nu)\big)(X \cdot\nu) \Big).
\end{equation}
Along $\pO$ we can decompose the divergence of $X$ into tangential and normal components, obtaining
\[
	\dv X = \dv_{\pO} X + g(\nabla_\nu X, \nu) =  \dv_{\pO} T + H(X\cdot\nu) + g(\nabla_\nu X, \nu),
\]
and so
\eqref{bdry1} becomes
\begin{align}
\begin{split}\label{bdry2}
	W\cdot\nu &= 2 w \frac{\p w}{\p\nu} + 2 \frac{\p\psi}{\p\nu} \nabla_T \left( (X\cdot\nu) \frac{\p\psi}{\p\nu}\right) \\
	&\qquad + \left(\frac{\p\psi}{\p\nu}\right)^2 \Big(- (X' \cdot\nu) + \big(\dv_{\pO} T - g(\nabla_\nu X, \nu)\big)(X \cdot\nu) + H(X\cdot\nu)^2  \Big) .
\end{split}
\end{align}
Finally, we observe that
\[
	2 \frac{\p\psi}{\p\nu} \nabla_T \left( (X\cdot\nu) \frac{\p\psi}{\p\nu}\right) + \left(\frac{\p\psi}{\p\nu}\right)^2 (X\cdot\nu)\dv_{\pO} T 
	= \dv_{\pO} \left[ \left(\frac{\p\psi}{\p\nu}\right)^2 (X\cdot\nu) T \right] + \left(\frac{\p\psi}{\p\nu}\right)^2 \nabla_T(X\cdot\nu)
\]
and hence
\[
	W \cdot\nu = 2 w \frac{\p w}{\p\nu}  + \left(\frac{\p\psi}{\p\nu}\right)^2 \Big(- (X' \cdot\nu) - g(\nabla_\nu X, \nu)(X \cdot\nu) + \nabla_T(X\cdot\nu) + H(X\cdot\nu)^2  \Big) + \dv_{\pO}(\ast).
\]
Integrating over $\pO$ completes the proof, since $\int_{\pO}\dv_{\pO}(\ast)\,d\mu = 0$.
\end{proof}


\section{Computing the Hessian at a critical partition}
\label{sec:Hessian}
We are now ready to characterize the critical points of $\lap$ and compute its Hessian, thus proving \Cref{thm:submanifold,thm:critical,thm:Hess1}.

\subsection{Equipartitions: proof of \Cref{thm:submanifold}}
\label{ssec:eq}
Assuming $P = \{\Omega_i\}_{i=1}^k$ is an equipartition with corners, we now show that $\EsM \subset \DsM$ is a smooth submanifold, thus proving \Cref{thm:submanifold}. The proof is almost identical to that of \cite[Proposition~8]{BKS12}, so we just give a summary.

Consider the smooth map
\begin{equation}
	\Xi \colon \DsM \to \bbR^k, \qquad \Xi(\varphi) = \Big(  \lambda_1\big(\varphi(\Omega_1)\big), \ldots, \lambda_1\big(\varphi(\Omega_k)\big) \Big),
\end{equation}
defined so that $\EsM \subset \DsM$ is the preimage of the diagonal in $\bbR^k$. The subdomains $\Omega_i$ are regular enough for the ``standard" boundary version of Hadamard's formula to apply, as was shown in \Cref{thm:Hadamard}. Therefore, one can use this formula, exactly as in \cite{BKS12}, to prove that $\Xi$ is transversal to the diagonal in $\bbR^k$ and hence its preimage is a smooth submanifold of codimension $k-1$.

It follows that $\lap(\varphi) = \Lambda\big(\varphi(P)\big)$ is a smooth function on $\EsM$, since each $\lambda_1\big(\varphi(\Omega_i)\big)$ depends smoothly on $\varphi$. Since $\lambda_1\big(\varphi(\Omega_i)\big)$ does not depend on $i$, the same is true of its derivative, so we can use Hadamard's formula to conclude that
\begin{equation}
\label{TEsM}
	\int_{\pO_1} (X \cdot \nu_1) \left(\frac{\p \psi_1}{\p \nu_1}\right)^2 d\mu = 
	\cdots = \int_{\pO_k} (X \cdot \nu_k)\left(\frac{\p \psi_k}{\p \nu_k}\right)^2 d\mu
\end{equation}
for all $X \in T_{\id} \EsM$.

\begin{rem}
\label{rem:criticalH}
Given $X \in T_{\id} \DsM$, if each of the integrals in \eqref{TEsM} vanishes, then $X \in T_{\id} \EsM$. Conversely, if $P$ is critical, then each of the integrals in \eqref{TEsM} vanishes for all $X \in T_{\id} \EsM$.
\end{rem}

\subsection{Critical partitions: proof of \Cref{thm:critical}}
Similar to the proof of \Cref{thm:submanifold} in \Cref{ssec:eq}, this proof only depends on Hadamard's formula and hence is identical to the proof of \cite[Theorem~9]{BKS12}. That is, one uses Lagrange multipliers to show that the identity is a critical point of $\lap$ if and only if it is a critical point of the functional
\[
	\lambda_{\mathbf c} \colon \DsM  \lra \bbR, \qquad \lambda_{\mathbf c}(\varphi) = \sum_{i=1}^k c_i \lambda_1\big(\varphi(\Omega_i)\big)
\]
for some choice of non-negative numbers $c_1, \ldots, c_k$ with $c_1 + \cdots + c_k = 1$. An application of Hadamard's formula, from \Cref{thm:Hadamard},  then shows that the identity is a critical point of $\lambda_{\mathbf c}$ if and only if \eqref{eq:normals} holds with $a_i^2 = c_i$.

\subsection{The Hessian: proof of \Cref{thm:Hess1}}
\label{ssec:pfHess1}
Next, we compute the Hessian at a critical partition. By the polarization identity, it suffices to compute the diagonal entries.

\begin{prop}
\label{prop:Hess}
If $P$ is a critical partition and $X \in T_{\id} \EsM$, then
\begin{equation}
\label{critHess}
	\Hess\lap(\id)(X,X) 
	= 2 \sum_{i=1}^k \int_{\Omega_i} \big(|\nabla u_i|^2 - \lambda u_i^2\big)\,dV,
\end{equation}
where $u_i \in H^1(\Omega_i)$ is the unique solution to
\begin{equation}
\label{ui:pde}
	\Delta u_i + \lambda u_i = 0, \qquad u_i\big|_{\pO_i} = \rho (X \cdot\nu_i), \qquad \int_{\Omega_i} u_i \psi_i \,dV = 0,
\end{equation}
with $\rho$ defined by \eqref{rhodef}.
\end{prop}

\begin{proof}[Proof of \Cref{prop:Hess}]
Fix $X \in T_{\id} \EsM$ and let $\varphi_t$ be a $C^2$ curve in $\EsM$ with $\varphi_0 = \id$ and $\varphi_0' = X$.
Since $\varphi_t$ is a family in $\EsM$, we have $\lap(\varphi_t) = \lambda_1(\varphi_t(\Omega_i))$ for each $i$, so \Cref{interior} gives
\begin{equation}
\label{proof:Hess1}
	\Hess\lap(\id)(X,X)  = \frac{d^2 \lap(\varphi_t)}{dt^2}\bigg|_{t=0} = \int_{\Omega_i} \dv W_i \,dV.
\end{equation}
Here $W_i = X(\dv Y_i) +  \tilde Y_i + 2Z_i$ with $Y_i$, $\tilde Y_i$ and $Z_i$ defined as in \eqref{Ydef}, \eqref{Ytildedef} and \eqref{Zdef}, with $\psi_i$ in place of $\psi$ and with $w_i \in H^1(\Omega_i)$ denoting the unique solution to
\begin{equation}
\label{wi:pde}
	\Delta w_i + \lambda w_i = 0, \qquad w_i\big|_{\pO_i} = - (X \cdot\nu_i) \frac{\p\psi_i}{\p\nu_i}, \qquad \int_{\Omega_i} w_i \psi_i \,dV = 0.
\end{equation}
Choosing $\{a_i\}$ as in \Cref{thm:critical}, we can rewrite \eqref{proof:Hess1} as
\begin{equation}
\label{proof:Hess2}
	\Hess\lap(\id)(X,X) = \sum_{i=1}^k a_i^2 \int_{\Omega_i} \dv W_i\, dV.
\end{equation}

We next use the fact that $\dv(w_i\nabla w_i) = |\nabla w_i|^2 - \lambda w_i^2$ to write
\[
	\dv W_i = 2\big( |\nabla w_i|^2 - \lambda w_i^2\big) + \dv\underbrace{\big[ X(\dv Y_i) + \tilde Y_i + 2(Z_i - w_i \nabla w_i) \big]}_{ = \tilde W_i}.
\]
Depending on the signs of $a_i$ and $\psi_i$, either $a_i w_i$ or $-a_i w_i$ will be the unique solution to \eqref{ui:pde}, and so either way we have
\begin{equation}
\label{proof:Hess3}
	\Hess\lap(\id)(X,X) = 2 \sum_{i=1}^k \int_{\Omega_i} \big(|\nabla u_i|^2 - \lambda u_i^2\big)\,dV
	+ \sum_{i=1}^k a_i^2 \int_{\Omega_i} \dv \tilde W_i\, dV.
\end{equation}

To complete the proof we need to show that $\sum a_i^2 \int_{\Omega_i} \dv \tilde W_i\, dV = 0$. 
We use \Cref{lem:H1approx}  to choose a sequence of functions $\{\phi_n\}$ in $C^\infty(M)$, each vanishing near the corner points of $\pO$, with $\phi_n \to 1$ in $H^1(M^o)$. Since each $\pO_i$ is piecewise smooth, we get that $\psi_i$ and $w_i$ are $H^3$ and $H^2$ up to the boundary, except near the corners. It follows that $\phi_n \tilde W_i \in W^{1,1}(\Omega_i)$, so we can apply the divergence theorem to obtain
\begin{equation}
\label{Wdiv1}
	\int_{\Omega_i} \big(\phi_n \dv \tilde W_i + \nabla_{\tilde W_i} \phi_n\big)\,dV = \int_{\pO_i} \phi_n (\tilde W_i \cdot \nu_i) \,d\mu
	= \int_{\Sigma_i} \phi_n (\tilde W_i \cdot \nu_i) \,d\mu.
\end{equation}
Evaluating the boundary terms (cf. the proof of \Cref{thm:2var}), we get
\begin{equation}
\label{Wbdry2}
	\tilde W_i \cdot\nu = \left(\frac{\p\psi_i}{\p\nu_i}\right)^2 \Big(- (X' \cdot\nu_i) + \big(\dv X - 2 \nabla X(\nu_i,\nu_i)\big)(X \cdot\nu_i)  \Big) + 2 \frac{\p\psi_i}{\p\nu_i} \nabla_T \left( (X\cdot\nu_i) \frac{\p\psi_i}{\p\nu_i}\right)
\end{equation}
away from the corners, where $T = X - (X\cdot\nu)\nu$ is the tangential part of $X$. Since $P$ is a critical partition, \Cref{thm:critical} gives
\[
	a_i^2 \left(\frac{\p \psi_i}{\p \nu_i} \right)^2 = a_j^2 \left( \frac{\p \psi_j}{\p \nu_j} \right)^2
\]
on each $\Sigma_i \cap \Sigma_j$. Using $\nu_i = - \nu_j$ for $i \neq j$, we conclude from \eqref{Wbdry2} that $a_i^2 (\tilde W_i \cdot \nu_i) = - a_j^2 (\tilde W_j \cdot \nu_j)$, so from \eqref{Wdiv1} we get
\[
	\sum_{i=1}^k a_i^2 \int_{\Omega_i} \big(\phi_n \dv \tilde W_i + \nabla_{\tilde W_i} \phi_n\big)\,dV = 0
\]
because the boundary terms cancel in pairs. Finally, we observe that $\psi_i \in C^1(\bar\Omega_i)$, by \cite[Theorem~1]{Azzam1}. This and the fact that $w_i \in H^1(\Omega_i)$ together imply $\tilde W_i$ and $\dv \tilde W_i$ are both in $L^2(\Omega_i)$. Since $\phi_n \to 1$ in $H^1(\Omega_i)$, we get
\begin{equation}
	\sum_{i=1}^k a_i^2 \int_{\Omega_i} \dv \tilde W_i\, dV = \lim_{n \to \infty} \sum_{i=1}^k a_i^2 \int_{\Omega_i} \big(\phi_n \dv \tilde W_i + \nabla_{\tilde W_i} \phi_n\big)\,dV = 0,
\end{equation}
as was to be shown.
\end{proof}

We are now ready to prove our main formula for the Hessian of $\lap$.

\begin{proof}[Proof of \Cref{thm:Hess1}]
Fix $i$ and let $u_i \in H^1(\Omega_i)$ denote the unique solution to \eqref{ui:pde}. Since $\nu_i = \chi_i \nu$
and $X \cdot\nu\big|_{\pO_i\backslash\Sigma_i} = 0$ for $X \in T_{\id} \EsM$, we conclude that
\begin{equation}
	u_i\big|_{\Sigma_i} = \chi_i \rho (X \cdot\nu), \qquad u_i\big|_{\pO_i \setminus \Sigma_i} = 0,
\end{equation}
therefore $u_i$ is the unique solution to \eqref{uiequation}, with $f = \rho (X \cdot\nu)$, for which $\int_{\Omega_i} u_i \psi_i \,dV = 0$. From the definition of $\form$ in \eqref{aformdef} we then have
\begin{equation}
	\form\big(\rho(X\cdot\nu), \rho(X \cdot\nu) \big) = \sum_{i=1}^k \int_{\Omega_i} \big(|\nabla u_i|^2 - \lambda u_i^2\big)\,dV,
\end{equation}
which is half of the right-hand side of \eqref{critHess}.
\end{proof}


\section{Comparing indices}
\label{sec:comparing}

Our main result, \Cref{thm:equality}, is an immediate consequence of \Cref{thm:BCHS} (which was proved in \cite{BCHS}) and \Cref{thm:index}, which we prove in this section. The proof relies on a technical approximation result. Letting $F$ denote the corner points of $P$ and recalling that $M^o$ denotes the interior of $M$, we define
\begin{equation}
\label{T0infty}
	\cT_0^\infty = \big\{ X \in T_{\id} \EsM : X \in \mathfrak{X}(N) \text{ and } \supp X \subset M^o \backslash F \big\},
\end{equation}
that is, the set of smooth vector fields that vanish near $\p M$ and all of the corner points of $P$.

\begin{theorem}
\label{thm:dense}
Let $P$ be a critical partition and fix a unit normal vector field $\nu$ along its boundary set, as in  \Cref{thm:Hess1}. Then $\{ \rho(X\cdot\nu) : X \in \cT_0^\infty \}$ is a dense subset of $\dom(\form)$.
\end{theorem}

\Cref{thm:index} is a consequence of the weaker (but still nontrivial) result that $\{ \rho(X\cdot\nu) : X \in T_{\id}  \EsM \}$ is dense in $\dom(\form)$. The full strength of \Cref{thm:dense} is needed, however, to prove \Cref{thm:nonbi}. See \Cref{sec:proofnonbi}, in particular the proof of \Cref{prop:Hessmin}.

\begin{rem}
Letting $\mathcal F = \{ \rho(X\cdot\nu) : X \in T_{\id}  \EsM \}$, we note the following.
\begin{enumerate}
	\item $\mathcal F$ is a proper subset of $\dom(\form)$ even when $\Sigma$ is smooth. This is because functions in $\mathcal F$ are of class $H^{s-1/2}$ for some fixed $s>3$, whereas functions in $\dom(\form)$ are only $H^{1/2}$.
	\item When $\Sigma$ is smooth, $\mathcal F$ is large enough to contain all of the eigenfunctions of $\DtnNu$. This was used in \cite[Theorem~6]{BCCM2} to prove \Cref{thm:index} for smooth partitions.
	\item When $P$ has corners, every function in $\mathcal F$ vanishes at the corner points (because $\rho$ does). As a result, the eigenfunctions of $\DtnNu$ are not necessarily contained in $\mathcal F$, but they can be approximated by functions in $\mathcal F$, using \Cref{thm:dense}.
\end{enumerate}
\end{rem}

The argument is rather delicate, since we are approximating functions in $\dom(\form)$ with functions that vanish near the corner points. This is possible because $\dom(\form)$ is equipped with the $H^{1/2}$ norm; no such approximation exists in $H^{1/2+\epsilon}$ for $\epsilon>0$.

In \Cref{sec:equality} we explain how \Cref{thm:dense} implies \Cref{thm:index}. In \Cref{sec:approx} we establish some approximation results and in \Cref{sec:inj} we use them to prove \Cref{thm:dense}.

\subsection{Equality of Morse indices: proof of \Cref{thm:index}}
\label{sec:equality}

In the proof we will not refer to $\DtnNu$ directly, but instead work with the bilinear form $\form$ that was defined in \eqref{aformdef}. Recall the formula
\begin{equation}
\label{Hess3}
	\Hess \lap(\id)(X_1,X_2) = 2\form\big(\rho(X_1\cdot\nu), \rho(X_2 \cdot\nu) \big), \qquad X_1, X_2 \in T_{\id} \EsM,
\end{equation}
from \Cref{thm:Hess1}. If $\Hess \lap(\id)$ is negative on $\spn \{X_1, \cdots, X_m\} \subset T_{\id} \EsM$, then $\form$ is negative on $\spn\{\rho(X_1 \cdot\nu), \ldots, \rho (X_m \cdot\nu)\}$. We claim that this space is $m$ dimensional. If not, there would exist a nontrivial linear combination $X = c_1 X_1 + \cdots + c_m X_m$ for which $\rho(X \cdot \nu) = 0$, but then \eqref{Hess3} would imply $\Hess \lap(\id)(X,X) = 0$, contradicting the fact that $\Hess \lap(\id)$ is negative on $\spn \{X_1, \cdots, X_m\}$. It follows that $n_-(\Hess \lap) \leq n_-(\DtnNu)$.

The reverse inequality, $n_-(\DtnNu) \leq n_-(\Hess \lap)$, is proved using a density argument. Suppose $\form$ is negative on $\spn\{f_1, \ldots, f_m\}$, so the $m\times m$ matrix with entries $\form (f_i,f_j)$ is negative definite. We use \Cref{thm:dense} to approximate each $f_i$ by $h_i = \rho(X_i \cdot\nu)$ for some $X_i \in T_{\id} \EsM$. By choosing each $h_i$ sufficiently close to $f_i$, we can ensure that $h_1, \ldots, h_m$ (and hence $X_1, \ldots, X_m$) are linearly independent and the matrix $\form( h_i, h_j)$ is negative definite, since \eqref{abound} says that $\form$ is continuous on $\wtH(\Sigma)$. It follows from \eqref{Hess3} that
$
	\Hess\lap(\id)(X_i, X_j) = 2 \form\big( h_i, h_j \big),
$
therefore $\Hess\lap(\id)$ is negative on the span of $X_1, \ldots, X_m$. Since the $X_i$ are linearly independent, this implies $n_-(\DtnNu) \leq n_-(\Hess \lap)$ and completes the proof of \Cref{thm:index}.

\begin{rem}
In fact, we get even more from \Cref{thm:dense}. Not only is the Morse index of $\DtnNu$ equal to the Morse index of the Hessian, it is equal to the Morse index of the Hessian \emph{restricted to the subspace} $\cT^\infty_0 \subset T_{\id} \EsM$; cf. \Cref{prop:Hessmin}.
\end{rem}

\subsection{Approximation results}

\label{sec:approx}

Recall the Hilbert space $H^1_0(M;\p P)$ from \eqref{H10P}, with the norm $\| \cdot \|_{_P}$ defined in \eqref{H10Pnorm}. For convenience we set $\Sigma = \p P$, so we denote this space by $H^1_0(M;\Sigma)$.
Our goal in this section is the following density result. 

\begin{theorem}
\label{thm:approx}
Let $F \subset M$ denote the corner points of $P$. The set
\begin{equation}
	\CF = \big\{ \phi \in H^1_0(M; \Sigma) :  \supp \phi \subset M^o \backslash F \text{ and }\phi_i \in C^\infty(\bar\Omega_i) \text{ for each $i$} \big\}
\end{equation}
is dense in $H^1_0(M; \Sigma)$.
\end{theorem}

That is, we can approximate any $u \in H^1_0(M; \Sigma)$ by a function that is smooth up to the boundary of each $\Omega_i$, satisfies anti-continuity conditions along all pairwise intersections $\Sigma_i \cap \Sigma_j$, and vanishes near the boundary of $M$ and the corner points of $P$.

An easy consequence of this theorem is that we can approximate functions in $\wtH(\Sigma)$ by smooth functions supported away from the singular part of the nodal set. This will be used below in the proof of \Cref{thm:dense}.

\begin{cor}
\label{cor:approxSigma}
$C^\infty_0(\Sigma \backslash F)$ is dense in $\wtH(\Sigma)$.
\end{cor}

\begin{proof}
Given $f \in \wtH(\Sigma)$, there exists $u \in H^1_0(M; \Sigma)$ with $u_i\big|_{\Sigma_i} = \chi_i f_i$ for each $i$; cf. \cite[Lemma~2.1]{BCHS}. Using \Cref{thm:approx}, we approximate $u$ by $\phi \in \CF$. The fact that $\phi$ satisfies anti-continuity conditions along $\Sigma$ guarantees that $\tilde f\big|_{\Sigma_i} := \chi_i \phi_i\big|_{\Sigma_i}$ gives a well defined function $\tilde f \in C^\infty_0(\Sigma \backslash F)$. Since each trace map $H^1(\Omega_i) \to H^{1/2}(\pO_i)$ is bounded, we have
\[
	\| \chi_i \tilde f_i - \chi_i f_i \|_{H^{1/2}(\pO_i)} \leq C \| \phi_i - u_i \|_{H^1(\Omega_i)}
\]
for each $i$, and so $\|f - \tilde f\|_{\wtH(\Sigma)} \leq C \|u - \phi\|_{_P}$ can be made arbitrarily small.
\end{proof}

The first step in the proof of \Cref{thm:approx} is to show that $H^1$ functions can be approximated by smooth functions that vanish near a finite set of points.

\begin{lemma}
\label{lem:H1approx}
If $U \subset M$ is a bounded, Lipschitz domain and $F \subset \bar U$ is finite, then
\[
	\big\{ \phi \in C^\infty(\bar U) : \phi \text{ vanishes in a neighborhood of $F$} \big\}
\]
is dense in $H^1(U)$.
\end{lemma}

\begin{proof}
Using a partition of unity and local coordinates, it suffices to prove the result when $U \subset \bbR^2$ is a bounded, Lipschitz domain and $F = \{x_0\} \subset \bar U$.

We first claim that $C^\infty(\bbR^2 \backslash \{x_0\})$ is dense in $H^1(\bbR^2)$. By \cite[Theorem~3.28]{Adams}, it suffices to prove that the only distribution $\ell \in H^{-1}(\bbR^2)$ with $\supp \ell \subset \{x_0\}$ is $\ell=0$. If $\ell$ is a distribution with $\supp \ell \subset \{x_0\}$, it must be of the form $\ell = \sum_{|\alpha| \leq m} a_\alpha \partial^\alpha \delta_{x_0}$ for some finite $m$, where $\{a_\alpha\}$ are constants and $\delta_{x_0}$ is the Dirac delta distribution at $x_0$; see, for instance, \cite[Theorem~3.9]{M00}.
Recalling that $\ell \in H^{-1}(\bbR^2)$ if and only if its Fourier transform satisfies $(1+|\xi|^2)^{-1/2} |\hat\ell(\xi)| \in L^2(\bbR^2)$, we compute
\[
	\hat\ell(\xi) = e^{-2\pi i (x_0 \cdot \xi)} \sum_{|\alpha| \leq m} a_\alpha (2\pi i \xi)^\alpha .
\]
It follows that $\ell \in H^{-1}(\bbR^2)$  if and only if all $a_\alpha = 0$, in which case $\ell=0$.

Since $U \subset \bbR^2$ is Lipschitz, there is an extension operator $E \colon H^1(\Omega) \to H^1(\bbR^2)$; see, for instance, \cite[Appendix~A]{M00}. Given $u \in H^1(U)$ and $\epsilon > 0$, we use the above claim to find $\tilde u \in C^\infty(\bbR^2 \backslash \{x_0\})$ with $\|\tilde u - Eu\|_{H^1(\bbR^2)} < \epsilon$, and hence $\| \tilde u|_U - u \|_{H^1(U)} \leq \| \tilde u - Eu \|_{H^1(\bbR^2)} < \epsilon$. Thus $\phi = \tilde u|_U$ is the desired approximation.
\end{proof}

The next lemma will help us analyze $H^1_0(M;\Sigma)$ functions near points where nodal lines intersect.

\begin{lemma}
\label{unfold}
Let $D = \{(x,y) \in \bbR^2 : x^2 + y^2 < 1\}$, $D_+ = \{(x,y) \in D : x > 0\}$ and $\Gamma = [-1,0] \times \{0\}$. The function $D \backslash\Gamma \to D_+$ given in polar coordinates by $(r,\theta) \mapsto (r,\theta/2)$ induces a bounded linear map $T \colon H^1(D \backslash\Gamma) \to H^1(D_+)$ with bounded inverse.
\end{lemma}

\begin{figure}
\begin{tikzpicture}[ scale=0.8]
	\draw[thick,dashed] (2,0) arc[radius=2, start angle=0, end angle=360];
	\draw[very thick] (0,0) -- (-2,0);
\end{tikzpicture}
\hspace{3cm}
\begin{tikzpicture}[ scale=0.8]
	\draw[thick,dashed] (0,-2) arc[radius=2, start angle=-90, end angle=90];
	\draw[very thick] (0,-2) -- (0,2);
\end{tikzpicture}
\caption{The slit disk $D \backslash \Gamma$ (left) and half disk $D_+$ (right) from \Cref{unfold}.}
\label{fig:unfolding}
\end{figure}

This is useful because $D_+$ is a Lipschitz domain whereas $D \backslash \Gamma$ is not; see \Cref{fig:unfolding}. We omit the proof, which follows from an explicit computation of the $H^1$ norms in $D \backslash\Gamma$ and $D_+$.

We are now ready to prove our main approximation theorem.

\begin{proof}[Proof of \Cref{thm:approx}]
Since $P$ is a partition with corners, we can cover $M$ by a finite number of open sets $\{U_a\}$ in $N$ for which there exist coordinate charts $\psi_a \colon U_a \to D$ each satisfying one of the following:
\begin{enumerate}
	\item $D \cap \psi_a(\Sigma) = \varnothing$;
	\item $D \cap \psi_a(\Sigma) = (-1,1) \times \{0\}$;
	\item $D \cap \psi_a(\Sigma)$ is a finite number of line segments meeting at the origin.
\end{enumerate}
(Cf. \Cref{fig:nbhd}, which shows a covering of $\p M \cup \Sigma$ by sets of type (2) and (3).)

Let $\{\eta_a\}$ be a partition of unity subordinate to this cover. Given $u \in H^1_0(M;\Sigma)$ and $\epsilon>0$, it suffices to approximate each $\eta_a u$ by a function $\phi_a \in \CF$ having $\supp \phi_a \subset U_a$ and $\|\eta_a u - \phi_a\|_{_P} < \epsilon$.

\underline{Case 1}: If $U_a$ does not intersect $\Sigma$, then $\eta_a u \in H^1_0(U_a)$, and the existence of $\phi_a$ follows from the density of $C^\infty_0(U_a)$ in $H^1_0(U_a)$.

\underline{Case 2}: In this case, $U_a$ intersects two subdomains, say $\Omega_i$ and $\Omega_j$, and $\eta_a u$ is anti-continuous across their common boundary, $\Sigma_i \cap \Sigma_j \cap U_a$. Consider the function $\hat u_a$ that equals $\eta_a u$ on $U_a \cap \Omega_i$ and $-\eta_a u$ on $U_a \cap \Omega_j$. By construction, $\hat u_a \in H^1_0(U_a)$, so we can approximate it in the $H^1_0(U_a)$ norm by some $\hat \phi_a \in C^\infty_0(U_a)$. It follows that the function $\phi_a$ that equals $\hat \phi_a$ on $U_a \cap \Omega_i$ and $-\hat \phi_a$ on $U_a \cap \Omega_j$ is contained in $H^1_0(M;\Sigma)$ and approximates $\eta_a u$ in the $\| \cdot \|_{_P}$ norm.

\underline{Case 3}: In this case, we can find local coordinates in which $U_a$ corresponds to the unit disk and $\Sigma$ is a finite number of line segments intersecting at the origin. If the number of line segments is even, then we can change the sign of $\eta_a u$ on every other subdomain (as in Case 2) to obtain a function in $H^1(U_a)$, and then apply \Cref{lem:H1approx}.

If the number of line segments is odd, then we can reduce to the situation where an anti-continuity condition is satisfied on just one segment of $\Sigma$, which we take to be $[-1,0] \times \{0\}$. Applying the ``unfolding" transformation of \Cref{unfold} to $\eta_a u$, we get a function $\tilde u \in H^1(D_+)$ that satisfies $\tilde u(0,-y) = -\tilde u(0,y)$ (in the trace sense) on the $y$-axis. Moreover, $\tilde u$ vanishes on the rest of the boundary, since $\supp \eta_a u \subset D$.

Noting that $\tau(x,y) = (x,-y)$ induces an isomorphism on $H^1(D_+)$, we decompose $\tilde u$ into even and odd parts,
\[
	\tilde u = \underbrace{\tfrac12( \tilde u + \tilde u \circ \tau)}_{\tilde u_{\rm e}} + \underbrace{\tfrac12( \tilde u - \tilde u \circ \tau)}_{\tilde u_{\rm o}}.
\]
By the antisymmetry condition, $\tilde u_{\rm e}$ vanishes on the $y$-axis, and hence is in $H^1_0(D_+)$, so there exists $\tilde v \in C^\infty_0(D_+)$ with $\|\tilde u_{\rm e} - \tilde v\|_{H^1(D_+)} < \epsilon$. On the other hand, by \Cref{lem:H1approx} there exists $\phi \in C^\infty(\bar D_+)$ vanishing in a neighborhood of the origin with $\| \tilde u - \phi \|_{H^1(D_+)} < \epsilon$ and hence $\|\tilde u_{\rm o} - \phi_{\rm o}\|_{H^1(D_+)} < \epsilon$. It follows that $\|\tilde u - (\tilde v + \phi_{\rm o}) \|_{H^1(D_+)} < 2\epsilon$, so $\tilde v + \phi_{\rm o}$ is the desired approximation.
\end{proof}

\subsection{Proof of \Cref{thm:dense}}
\label{sec:inj}

Since $P$ is a critical partition, the tangent space $T_{\id} \EsM$ consists of vector fields $X \in \mathfrak{X}^s(N)$ that are tangent to the smooth part of $\p M$ and satisfy
\begin{equation}
\label{tan0}
	\int_{\pO_i} (X \cdot \nu_i) \left(\frac{\p \psi_i}{\p \nu_i}\right)^2 d\mu = 0
\end{equation}
for $1 \leq i \leq k$; see \Cref{rem:criticalH}.

The following lemma is our main tool in the proof of \Cref{thm:dense}.

\begin{lemma}
\label{lem:fsmooth}
Let $P$ be a critical partition with boundary set $\Sigma$. Given $f \in \wtH(\Sigma)$ and $\epsilon>0$, there exists $h \in C^\infty_0(\Sigma\backslash F)$ such that $\|f - h \|_{\wtH(\Sigma)} < \epsilon$ and 
\begin{equation}
\label{ftildeint}
	\int_{\Sigma_i \cap \Sigma_j} f \frac{\p \psi_i}{\p \nu_i}\,d\mu = \int_{\Sigma_i \cap \Sigma_j} h \frac{\p \psi_i}{\p \nu_i}\,d\mu
\end{equation}
for all $i \neq j$.
\end{lemma}

\begin{proof}
Using \Cref{cor:approxSigma}, we can find $f_0 \in C^\infty_0(\Sigma\backslash F)$ with $\|f - f_0 \|_{\wtH(\Sigma)} < \epsilon$. To complete the proof we modify $f_0$ so that \eqref{ftildeint} holds. We do this by adding a small bump function supported in the interior of each nodal segment $\Sigma_i \cap \Sigma_j$.

Fix $i < j$ and choose a function $\eta \in C^\infty_0(\Sigma_i \cap \Sigma_j)$ that is nonnegative but not identically zero. Since $\frac{\p \psi_i}{\p \nu_i}$ only vanishes at the corners, this guarantees
$
	\int_{\Sigma_i \cap \Sigma_j} \eta \frac{\p \psi_i}{\p \nu_i}\,d\mu \neq 0.
$
Therefore we can choose a real number $s$ so that
\begin{equation}
\label{ftildeint2}
	\int_{\Sigma_i \cap \Sigma_j} f \frac{\p \psi_i}{\p \nu_i}\,d\mu = \int_{\Sigma_i \cap \Sigma_j} (f_0 + s\eta) \frac{\p \psi_i}{\p \nu_i}\,d\mu,
\end{equation}
namely
\[
	s = \frac{\int_{\Sigma_i \cap \Sigma_j} (f - f_0) \frac{\p \psi_i}{\p \nu_i}\,d\mu}{\int_{\Sigma_i \cap \Sigma_j} \eta \frac{\p \psi_i}{\p \nu_i}\,d\mu} .
\]

It follows that $|s| \leq C \|f - f_0 \|_{\wtH(\Sigma)}$ for some constant $C$ not depending on $f$ or $f_0$. Then the function $f_0 + s\eta$ satisfies \eqref{ftildeint2} on $\Sigma_i \cap \Sigma_j$, and
\[
	\big\| f - (f_0 + s\eta) \big\|_{\wtH(\Sigma)} \leq \left(1 + C \| \eta \|_{\wtH(\Sigma)} \right) \big\| f - f_0 \big\|_{\wtH(\Sigma)}.
\]

Repeating for all nodal segments, we obtain a function $h \in C^\infty_0(\Sigma\backslash F)$ such that $\|f - h \|_{\wtH(\Sigma)} < \epsilon$ and  \eqref{ftildeint} holds for all $i<j$. To complete the proof we must show that this also holds for $i>j$.

Since $P$ is a critical partition, we get from \Cref{thm:critical} that
\begin{equation}
	\frac{\p \psi_j}{\p \nu_j} = \left| \frac{a_i}{a_j}\right| \frac{\p \psi_i}{\p \nu_i}
\end{equation}
on $\Sigma_i \cap \Sigma_j$ (assuming each $\psi_i$ is positive, and hence has negative normal derivative). It follows that
\begin{equation}
	\int_{\Sigma_i \cap \Sigma_j} f \frac{\p \psi_j}{\p \nu_j}\,d\mu =	\left| \frac{a_i}{a_j}\right| \int_{\Sigma_i \cap \Sigma_j} f \frac{\p \psi_i}{\p \nu_i}\,d\mu = \left| \frac{a_i}{a_j}\right|  \int_{\Sigma_i \cap \Sigma_j} h \frac{\p \psi_i}{\p \nu_i}\,d\mu 
	= \int_{\Sigma_i \cap \Sigma_j} h \frac{\p \psi_j}{\p \nu_j}\,d\mu
\end{equation}
and so \eqref{ftildeint} holds for all $i \neq j$, as was to be shown.
\end{proof}

We are now ready to prove \Cref{thm:dense}, and hence complete the proof of \Cref{thm:index}.

\begin{proof}[Proof of \Cref{thm:dense}]
Let $f \in \dom(\form)$ and choose $h$ as in \Cref{lem:fsmooth}. Since $h$ vanishes in a neighborhood of each corner point (where $\rho$ is zero), $X = \rho^{-1} h \nu$ defines a smooth vector field on $\Sigma$, which can be extended to a vector field on $N$ with support in $M^o\backslash F$. By construction this satisfies $\rho(X\cdot\nu) = h$. For each $i$ we calculate
\begin{align*}
	\int_{\Sigma_i} \chi_i h \frac{\p \psi_i}{\p \nu_i} \,d\mu 
	= \sum_{j \neq i} \int_{\Sigma_i \cap \Sigma_j}  \chi_i h \frac{\p \psi_i}{\p \nu_i} \,d\mu 
	= \sum_{j \neq i} \int_{\Sigma_i \cap \Sigma_j} \chi_i f \frac{\p \psi_i}{\p \nu_i} \,d\mu 
	= \int_{\Sigma_i} \chi_i f \frac{\p \psi_i}{\p \nu_i} \,d\mu = 0,
\end{align*}
using \eqref{ftildeint} and the fact that $\chi_i$ is constant on $\Sigma_i \cap \Sigma_j$ for $i \neq j$. The last integral vanishes because $f$ is contained in $\dom(\form)$. Therefore $X$ satisfies \eqref{tan0} for each $i$ and hence is contained in $\cT^\infty_0$.
\end{proof}

\subsection{The nullity}
\label{sec:null}

Recall the submanifold 
\begin{equation}
	\DsP = \big\{ \varphi \in \DsM : \varphi(P) = P \big\}
\end{equation}
of diffeomorphisms that leave $P$ invariant, as described in \Cref{ssec:submanifold}. Its tangent space at the identity is the set of vector fields that are tangent to the smooth part of $\p M \cup \Sigma$. Since every such vector field has $X \cdot \nu = 0$, we see from \eqref{Hess1} that $\Hess\lap(X_1, X_2) = 0$ for any $X_1 \in T_{\id} \DsP$ and $X_2 \in T_{\id} \EsM$. This means the nullspace\footnote{The \emph{nullspace} of a bilinear form $b$ is the set of $x \in \dom(b)$ such that $b(x,y) = 0$ for all $y \in \dom(b)$, and the \emph{nullity} is the dimension of its nullspace.} of $\Hess\lap$ contains $T_{\id} \DsP$, and hence is infinite dimensional.

The nullity becomes finite when we restrict the Hessian to the normal bundle of $\DsP$, viewed as a submanifold of $\EsM$. To make this precise, we let $\big(T_{\id} \DsP\big)^\perp$ denote the $H^s$-orthogonal complement in $T_{\id} \EsM$, and denote by $n_0^\perp(\Hess\lap)$ the nullity of $\Hess\lap$ restricted to $\big(T_{\id} \DsP\big)^\perp$.

\begin{theorem}
\label{thm:null}
Under the conditions of \Cref{thm:equality}, we have
\begin{equation}
\label{eq:nullbound}
	n_0^\perp(\Hess\lap) \leq  n_0(\DtnNu) = n_0\big(\DP + \Lambda(P)\big) -1.
\end{equation}
\end{theorem}

\begin{proof}
The proof is similar to the first half of the proof of \Cref{thm:index} in \Cref{sec:equality}. Suppose the nullspace of the restricted Hessian contains linearly independent vector fields $\{X_1, \cdots, X_m\} \subset \big(T_{\id} \DsP\big)^\perp$. This means
\[
	2\form\big(\rho(X_i\cdot\nu), \rho(Y \cdot\nu) \big) = \Hess \lap(\id)(X_i,Y) = 0
\]
for all $Y \in \big(T_{\id} \DsP\big)^\perp$ and hence for all $Y \in T_{\id} \EsM$. It follows from \Cref{thm:dense} that 
$
	\form\big(\rho(X_i \cdot\nu), f \big) = 0
$
for all $f \in \dom(\form)$, therefore the nullspace of $\form$ contains $\spn\{\rho(X_1 \cdot\nu), \ldots, \rho (X_m \cdot\nu)\}$. This latter space is $m$ dimensional because the map
\begin{equation}
\label{Xmap}
	\big(T_{\id} \DsP\big)^\perp \ni X \mapsto \rho(X \cdot\nu) \in \dom(\form)
\end{equation}
is injective, therefore $n_0^\perp(\Hess\lap) \leq n_0(\DtnNu)$. On the other hand, from \cite[Theorem~1.7]{BCHS} we have $n_0(\DtnNu) = n_0\big(\DP + \Lambda(P)\big) -1$, and the result follows.
\end{proof}

If $\Lambda(P)$ is a simple eigenvalue of $-\DP$, it follows from \Cref{thm:null} that $n_0^\perp(\Hess\lap) = 0$, thus the restriction of the Hessian to the normal bundle is nondegenerate and \eqref{eq:nullbound} is an equality. If $\Lambda(P)$ is not simple, however, the inequality \eqref{eq:nullbound} can be strict: If there is a function in $\ker\DtnNu$ that does not vanish at a corner point, it will not be in the range of the map in \eqref{Xmap}, in which case $n_0^\perp(\Hess\lap) < n_0(\DtnNu)$.

\label{rem:null}A simple example is the 4-partition of the rectangle $(0,\alpha\pi) \times (0,\pi)$ generated by the eigenfunction $\psi_{2,2}(x,y) = \sin(2x/\alpha) \sin(2y)$, with $\alpha = \sqrt{5/3}$ chosen so that $\psi_{3,1}(x,y) = \sin(3x/\alpha) \sin(y)$ is an eigenfunction for the same eigenvalue. The restriction of $\psi_{3,1}$ to the nodal set of $\psi_{2,2}$ is in the kernel of $\DtnNu$. Since it is nonzero at the corner point $\big(\frac{\alpha\pi}{2}, \frac{\pi}{2} \big)$, we have $0 = n_0(\Hess\lap) < n_0(\DtnNu) = 1$.

\section{Local vs global minima}
\label{sec:LG}

\subsection{The bipartite case: proof of \Cref{thm:bi}}
Referring to the statement of the theorem, we will prove that (1) $\Rightarrow$ (2) $\Rightarrow$ (3) $\Rightarrow$ (1).

The implication (1) $\Rightarrow$ (2) was shown in \cite[Corollary~5.6]{HHOT}; see also \cite[Theorem~1.6]{BCCM3}. On the other hand, (2) $\Rightarrow$ (3) is trivial because a minimal partition is necessarily a local minimum. To prove (3) $\Rightarrow$ (1), we observe that if $P$ is a local minimum, then it is a critical partition for which $\Hess\lap$ is nonnegative. It then follows from \Cref{thm:equality BP} that $P$ is the nodal partition of a Laplacian eigenfunction with $\delta(P) = n_-(\Hess\lap) = 0$, therefore the eigenfunction is Courant sharp.

\subsection{The general case: proof of \Cref{thm:nonbi}}
\label{sec:proofnonbi}

As in the bipartite case, we will prove that (1) $\Rightarrow$ (2) $\Rightarrow$ (3) $\Rightarrow$ (1). The implication (1) $\Rightarrow$ (2) is exactly \cite[Theorem~1.6]{BCCM3} and (3) $\Rightarrow$ (1) follows from \Cref{thm:equality} as in the bipartite case. Therefore it remains to prove (2) $\Rightarrow$ (3).

As noted in the introduction, this proof is somewhat delicate in the non-bipartite case, since $\cP_k(P)$ only contains partition with the same corner points as $P$. As a result, the set $\{ \varphi \in \DsM : \varphi(P) \in \cP_k(P) \}$ does not contain a neighborhood of the identity, so it is not immediately clear that $P$ is a critical partition. 

As noted in \Cref{ssec:hom}, however, a small perturbation of $P$ \emph{will} have boundary set homologous to $\p P$ if the perturbation does not move any of the corner points in the interior of $M$. To make this precise, recall the function $\Phi$ defined in \eqref{Phi}, which maps a neighborhood $\cU \subset \mathfrak{X}^s_M(N)$ of the zero section onto a neighborhood of the identity in $\DsM$, by \Cref{lem:epsilon}.

\begin{lemma}
\label{def:hom}
Let $P$ be a partition with corners. If $U \in \cU$ vanishes at all of the corner points in $M^o$, then $\Phi(U)(P) \in \cP_k(P)$.
\end{lemma}

\begin{proof}
Let $\Gamma$ be a smooth segment of $\p P$. Since $U$ vanishes at the corner points, the one-parameter family of diffeomorphisms $\varphi_t = \Phi(tU)$ provides a fixed-endpoint homotopy between $\Gamma$ and $\Phi(U)(\Gamma)$, so they are homologous 1-chains, as defined in \cite{BCCM3}. This holds for every smooth curve in $\p P$, therefore the boundary sets of $P$ and $\Phi(U)(P)$ are homologous.
\end{proof}

Using this and the approximation result in \Cref{thm:dense}, we obtain the following result, which completes the proof of (2) $\Rightarrow$ (3) in \Cref{thm:nonbi}.

\begin{prop}
\label{prop:Hessmin}
If $P$ minimizes $\Lambda$ in $\cP_k(P)$, then $P$ is a critical partition with $\Hess \lap(\id)(X,X) \geq 0$ for \emph{all} $X \in T_{\id}\EsM$.
\end{prop}

\begin{proof}
Since $P$ is minimizing in $\cP_k(P)$, it is minimizing with respect to all diffeomorphisms that are sufficiently close to the identity and fix all of the corner points of $P$, by \Cref{def:hom}. Recalling the definition of $\cT_0^\infty$ in \eqref{T0infty}, it follows that $D\lap(\id)(X) = 0$ for all $X \in \cT_0^\infty$. It follows from \Cref{thm:dense} that $D\lap(\id)(X) = 0$ for all $X \in T_{\id}\EsM$, thus $P$ is a critical partition.

It similarly follows that $\Hess \lap(\id)(X,X) \geq 0$ for all $X \in \cT_0^\infty$. For arbitrary $X \in T_{\id}\EsM$ we can use \Cref{thm:dense} to find a sequence $\{X_n\}$ in $\cT_0^\infty$ such that $\rho(X_n \cdot \nu) \to \rho(X\cdot\nu)$ in $\wtH(\Sigma)$. Since $\Hess \lap(\id)(X_n,X_n) \geq 0$ for each $n$ and $\form$ is continuous in $\wtH(\Sigma)$, we can apply \Cref{thm:Hess1} twice to obtain
\begin{align*}
	\Hess \lap(\id)(X,X) &= 2\form\big(\rho(X\cdot\nu), \rho(X \cdot\nu) \big) \\
	&= \lim_{n \to \infty} 2\form\big(\rho(X_n \cdot\nu), \rho(X_n \cdot\nu) \big) \\
	&= \lim_{n \to \infty} \Hess \lap(\id)(X_n,X_n) \\
	& \geq 0,
\end{align*}
which completes the proof.
\end{proof}

\section{Examples}
\label{sec:examples}

We conclude the paper by showing how our techniques can be used to study non-bipartite minimal partitions. Starting with nodal partitions of the rectangle and disk that are not Courant sharp, we use \Cref{thm:Hess1} to explicitly compute negative directions for the Hessian. Sketching the resulting deformations, then gives insight into the structure of nearby minimal (or locally minimal) partitions, which we then investigate numerically.

\subsection{The $(2,2)$ mode on the rectangle}

Consider the rectangle $(0,\alpha\pi) \times (0,\pi)$, which has eigenfunctions and eigenvalues
\begin{equation}
	\psi_{m,n}(x,y) = \sin\big(\tfrac{m x}{\alpha}\big) \sin(ny), \qquad \lambda_{m,n} = \big(\tfrac{m}{\alpha}\big)^2 + n^2,
	\qquad m,n \in \bbN.
\end{equation}
In this section we apply our machinery to the 4-partition generated by $\psi_{2,2}$. It is easily verified that $\psi_{2,2}$ is Courant sharp, and hence the resulting partition is minimal, if and only if $\frac35 \leq \alpha^2 \leq \frac53$.

For $\frac53 < \alpha^2 < 4$ we have
\[
	\lambda_{1,1} < \lambda_{2,1} < \lambda_{1,2} < \lambda_{3,1} < \lambda_{2,2} < \cdots
\]
so $\lambda_{2,2}$ is simple and is the fifth eigenvalue of the Laplacian. The corresponding nodal partition thus has deficiency $\delta(P) = 1$, so it follows from \Cref{thm:BCHS} that $n_-(\DtnNu) = 1$. For $\alpha$ in this range we will identify the eigenfunction corresponding to the negative eigenvalue of $\DtnNu$ and then use \Cref{thm:Hess1} to sketch the resulting energy-decreasing deformation. We use this sketch to conjecture a minimal 4-partition for the rectangle, which is then investigated numerically. A similar analysis was carried out in \cite{BCCM2} for the $(3,1)$ mode on the square; the present example is more involved on account of the intersecting nodal lines.

We choose the unit normal $\nu$ to point outwards along the bottom-left and top-right subdomains, as shown on the left of \Cref{fig:rectangle}. For this choice $\DtnNu$ coincides with the regular two-sided Dirichlet-to-Neumann map.

\begin{theorem}
For $\frac53 < \alpha^2 < 4$, the operator $\DtnNu$ has a single negative eigenvalue, and the corresponding eigenfunction  is given by the even reflection of
\begin{equation}
\label{22DtNeigenfunction}
	f(x,y) = \begin{cases} \sin(\gamma_1 x) \sin\big(\tfrac{\gamma_2 \pi}2 \big), & 0 \leq x \leq \tfrac{\alpha\pi}{2}, \ y = \tfrac{\pi}2, \\
	\sin\big( \tfrac{\gamma_1 \alpha\pi}2) \sin(\gamma_2 y \big), & x = \tfrac{\alpha\pi}{2}, \ 0 \leq y \leq \tfrac{\pi}2,
	\end{cases}
\end{equation}
about the lines $x = \tfrac{\alpha\pi}{2}$ and $y = \tfrac{\pi}2$. Here $\gamma_1 \in (3/\alpha, 4/\alpha)$ and $\gamma_2 \in (1,2)$ are numbers satisfying
\begin{equation}
\label{gamma12A}
	\gamma_1^2 + \gamma_2^2 = \lambda_{2,2}
\end{equation}
and
\begin{equation}
\label{gamma12B}
	\gamma_1 \cot \left( \frac{\gamma_1 \alpha\pi}2 \right) = \gamma_2 \cot\left( \frac{\gamma_2 \pi}2 \right).
\end{equation}
\end{theorem}

\begin{figure}[!tbp]
	\begin{tikzpicture}[scale=0.7]
		\draw[thick, dashed] (0,0) -- (0,4) -- (6,4) -- (6,0) -- (0,0);
		\draw[very thick] (3,0) -- (3,4);
		\draw[very thick] (0,2) -- (6,2);
		\draw[->,very thick,blue] (3,1) -- (3.6,1);
		\draw[->,very thick,blue] (3,3) -- (2.4,3);
		\draw[->,very thick,blue] (1.5,2) -- (1.5,2.6);
		\draw[->,very thick,blue] (4.5,2) -- (4.5,1.4);
	\end{tikzpicture}
\hspace{1cm}
	\begin{tikzpicture}[scale=0.7]
		\draw[thick, dashed] (0,0) -- (0,4) -- (6,4) -- (6,0) -- (0,0);
		\draw[very thick] (3,0) to [out=120, in=240] (3,2);
		\draw[very thick] (3,2) to [out=60, in=300] (3,4);
		\draw[very thick] (0,2) to [out=45, in=225] (3,2);
		\draw[very thick] (3,2) to [out=45, in=225] (6,2);
	\end{tikzpicture}
\hspace{1cm}
	\begin{tikzpicture}[scale=0.7]
		\draw[thick, dashed] (0,0) -- (0,4) -- (6,4) -- (6,0) -- (0,0);
		\draw[very thick] (2.2,1.4) -- (3.8,2.6);
		\draw[very thick] (0,2) to [out=0, in=90] (3,0);
		\draw[very thick] (3,4) to [out=270, in=180] (6,2);
	\end{tikzpicture}
\caption{The $(2,2)$ partition of the rectangle (left), a qualitative illustration of an energy decreasing deformation (center), and the conjectured topology of a locally minimal 4-partition (right).}
\label{fig:rectangle}
\end{figure}
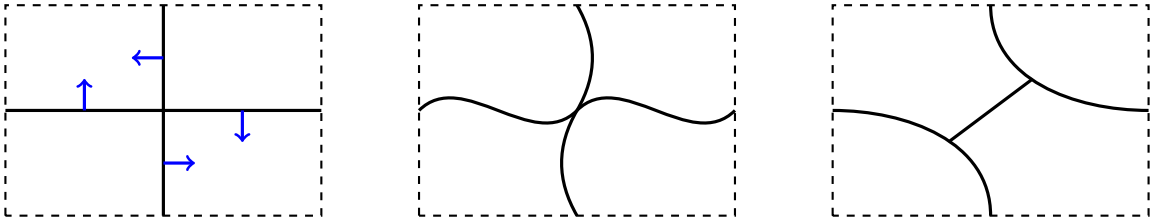

\begin{proof}
We determine the eigenfunction using the spectral flow method from \cite{BCM19}; see also \cite{beck2021limiting} for a detailed analysis of the one-dimensional case. It follows from \cite[Theorem~1]{BCM19} that there are smooth, strictly increasing functions $\gamma_1, \gamma_2 \colon [0,\infty) \to \bbR$ with
\[
	\gamma_1(0) = \frac3\alpha, \quad \gamma_1(\infty) = \frac4\alpha, \qquad \gamma_2(0) = 1, \quad \gamma_2(\infty) = 2
\]
such that the even reflection of
\begin{equation}
\label{usigma1}
	u_\sigma(x,y) = \sin\big(\gamma_1(\sigma) x)\big) \sin\big(\gamma_2(\sigma) y)\big)
\end{equation}
satisfies the two-sided boundary condition
\begin{equation}
	\label{22BVP}
	\frac{\p u_i}{\p \nu_i} + \frac{\p u_j}{\p \nu_j} + \sigma u = 0
\end{equation}
on $\Sigma$ for each $\sigma \geq 0$.

By assumption we have
\[
	\gamma_1(0)^2 + \gamma_2(0)^2 < \lambda_{2,2} < \gamma_1(\infty)^2 + \gamma_2(\infty)^2
\]
so there exists a unique $\bar\sigma \in (0,\infty)$ for which
\[
	\gamma_1(\bar\sigma)^2 + \gamma_2(\bar\sigma)^2 = \lambda_{2,2}.
\]
It follows from \eqref{22BVP} that the restriction of $u_{\bar\sigma}$ to the nodal set, which is precisely \eqref{22DtNeigenfunction}, is an eigenfunction of $\DtnNu$, with eigenvalue $-\bar\sigma$. Substituting \eqref{usigma1} into \eqref{22BVP} gives
\[
	2 \frac{\p u}{\p x}\left(\frac{\alpha\pi}2-, y\right) = 2 \gamma_1 \cos(\gamma_1 \alpha\pi/2) \sin(\gamma_2 y)
	= - \bar\sigma \sin(\gamma_1 \alpha\pi/2) \sin(\gamma_2 y), \qquad 0 \leq y \leq \tfrac{\pi}2,
\]
and
\[
	2 \frac{\p u}{\p y}\left(x, \frac\pi2 -\right) = 2 \gamma_2 \sin(\gamma_1 x) \cos(\gamma_2 \pi/2)
	= - \bar\sigma \sin(\gamma_1 x) \sin(\gamma_2 \pi/2), \qquad 0 \leq x \leq \tfrac{\alpha\pi}{2}.
\]
Solving both equations for $\bar\sigma$ and equating the resulting expressions gives \eqref{gamma12B}.
\end{proof}

We now sketch the resulting deformation for $\alpha = \frac32$. To obtain a valid deformation, we must approximate the eigenfunction $f$ in \eqref{22DtNeigenfunction} by a function $\tilde f$ that vanishes away from the corner points, then construct a vector field $X$ such that $\tilde f = \rho(X \cdot \nu)$, as in \Cref{thm:dense}. (This approximation is necessary\,---\,since $f$ does not vanish at the center of the rectangle, there is no vector field $X$ for which $f = \rho(X \cdot \nu)$.) However, we can understand the qualitative properties of the deformation, in particular, the direction in which it moves the nodal set, just by examining the sign of $f$.

For $\alpha = \frac32$ we have\footnote{
We can solve \eqref{gamma12A} and \eqref{gamma12B} numerically to find $\gamma_1 \approx 2.08$ and $\gamma_2 \approx 1.20$, but the exact values are not important. } $\gamma_1 \in (2, 8/3)$ and $\gamma_2 \in (1,2)$, therefore $\sin(\gamma_2 \pi/2) > 0$ and $\sin(\gamma_1 \alpha \pi/2) < 0$. It follows that $f(x,y)$ is negative for $x = \tfrac{\alpha\pi}{2}$ and $0 < y < \tfrac{\pi}2$. On the other hand, on the segment $0 < x < \tfrac{\alpha\pi}{2}, \ y = \tfrac{\pi}2$, $f(x,y)$ is positively proportional to $\sin(\gamma_1 x)$, and hence is positive near $x=0$ and negative near $x = \alpha \pi/2$, with a unique zero in between. Recalling the choice of $\nu$, we can sketch this deformation, as in the center of \Cref{fig:rectangle}.

This suggests that the energy will be decreased by pushing the nodal lines together in pairs, resulting in a non-bipartite partition with the topology shown on the right of \Cref{fig:rectangle}. Motivated by this, we look for critical partitions of this form.

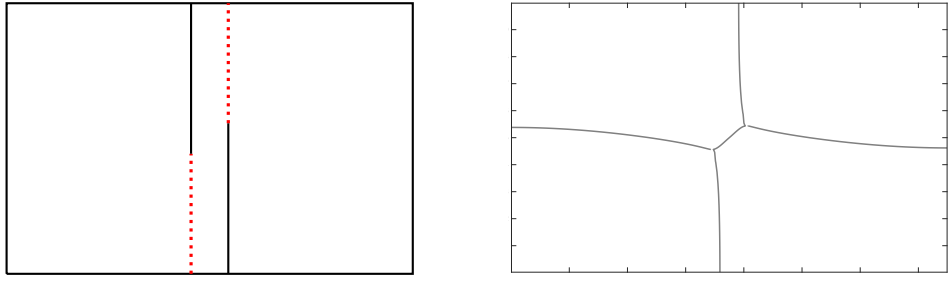
\begin{figure}
	\begin{tikzpicture}[scale=0.9]
		\draw[thick] (0,0) -- (0,4) -- (6,4) -- (6,0) -- (0,0);
		\draw[very thick, dotted, red] (2.725,0) -- (2.725,1.775);
		\draw[thick] (2.725,1.775) -- (2.725,4);
		\draw[very thick, dotted, red] (3.275,2.225) -- (3.275,4);
		\draw[thick] (3.275,0) -- (3.275,2.225);
	\end{tikzpicture}
\hspace{1cm}
\includegraphics[scale=0.285]{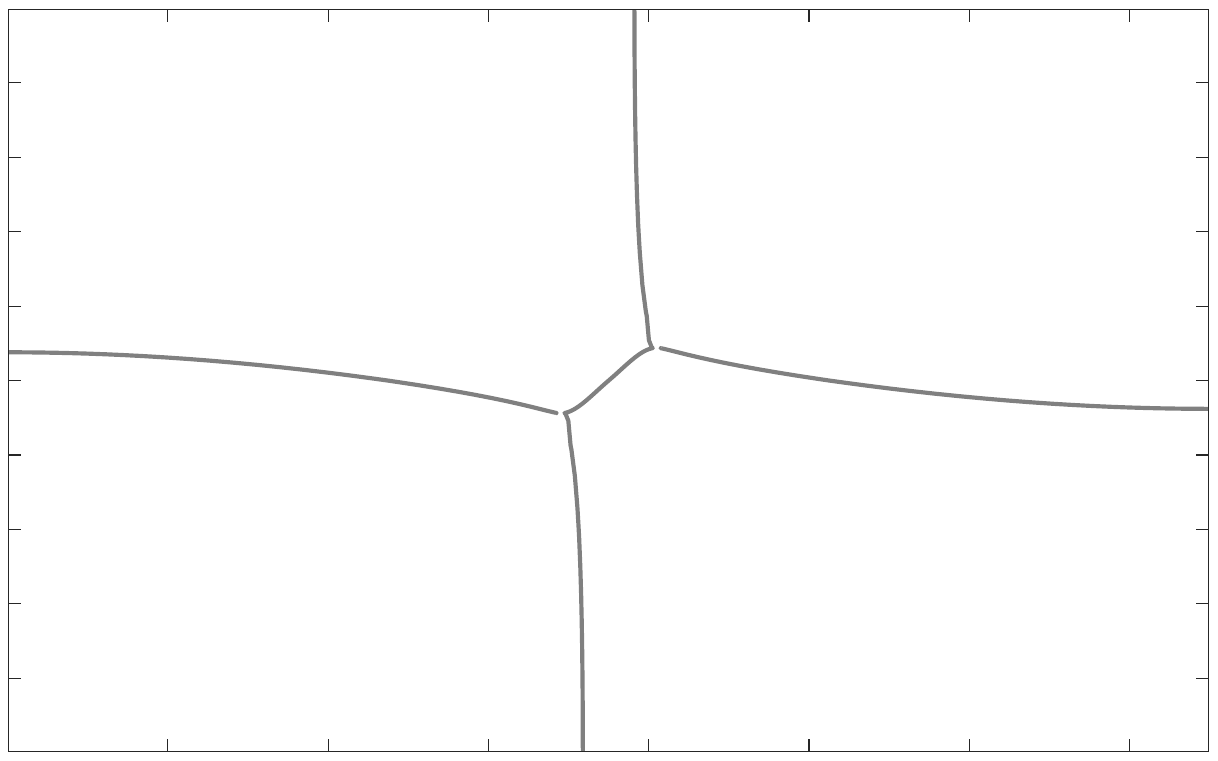}
\caption{A candidate locally minimal 4-partition of the rectangle with $\alpha = 3/2$ (right), corresponding to a Courant sharp eigenfunction of the operator with anti-continuity conditions imposed on the dotted red line segments (left).}
\label{fig:4partition}
\end{figure}

To do so, we consider the partition Laplacian with anti-continuity conditions on the two dotted red segments shown on the left of \Cref{fig:4partition}. For ease of computing we divide the domain into three subrectangles, imposing continuity conditions on the solid black vertical segments, which are not part of the cut set. Using separable coordinates, we discretize using Chebychev methods on the three rectangles. We then vary the endpoints of the dashed blue segments and examine the nodal partitions of the resulting eigenfunctions. For the cut shown in the figure, the fourth eigenfunction has nodal set shown on the right of \Cref{fig:4partition}.

\subsection{Radial partitions of the disk}

Consider the partition $P_k$ of the disk into $k$ equal sectors, as in \Cref{fig:disk}. This is the nodal partition of an eigenfunction for the corresponding partition Laplacian, with eigenvalue $\lambda_* = j_{k/2,1}^2$, where $j_{\alpha,n}$ denotes the $n$th positive zero of the Bessel function $J_\alpha$. The eigenfunction satisfies
\begin{equation}
	\psi(r,\theta)\big|_{\Omega_i} =  (-1)^{i+1} J_{k/2}\big(j_{k/2,1} r \big) \sin\left(\tfrac{k}2\right),
\end{equation}
where we have rotated the partition so that the segment $[0,1] \times \{0\}$ is contained in its nodal set. It follows from \Cref{thm:equality} that $P_k$ is a critical partition. However, it is only Courant sharp for $1 \leq k \leq 5$, thus for $k \geq 6$ the Hessian has nonzero Morse index, so $P_k$ is a saddle point of the partition energy functional, $\lap$.

Using \Cref{thm:Hess1}, we can construct a negative direction for $\Hess\lap$. We start with a convenient choice of unit normal $\nu$ along $\Sigma$. Labeling the subdomains $\Omega_1, \ldots, \Omega_k$ in the counterclockwise direction, we choose $\nu$ so that
\begin{equation}
\label{nu:disk}
	\nu\big|_{\Sigma_i} = (-1)^{i+1}\nu_i, \qquad 1 \leq i \leq k-1.
\end{equation}
This determines $\nu$. If $k$ is even, then $\nu\big|_{\Sigma_k} = - \nu_k$, while if $k$ is odd, then $\nu\big|_{\Sigma_k} $ will equal $-\nu_k$ on one segment of $\Sigma_k$ and $\nu_k$ on the other segment; see \Cref{fig:disk}.

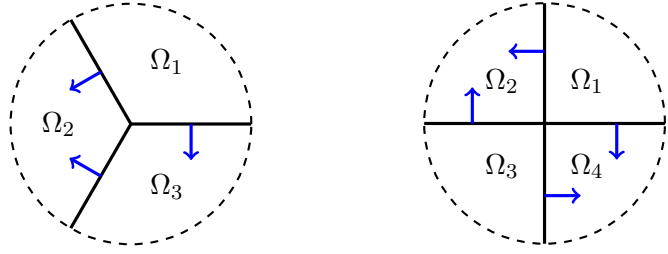
\begin{figure}
\begin{tikzpicture}[ scale=0.8]
	\draw[thick,dashed] (2,0) arc[radius=2, start angle=0, end angle=360];
	\draw[very thick] (0,0) -- (2,0);
	\draw[very thick] (0,0) -- ({2*cos(120)},{2*sin(120)});
	\draw[very thick] (0,0) -- ({2*cos(240)},{2*sin(240)});
	\node at ({1.2*cos(60)},{1.2*sin(60)}) {$\Omega_1$};
	\node at ({1.2*cos(180)},{1.2*sin(180)}) {$\Omega_2$};
	\node at ({1.2*cos(300)},{1.2*sin(300)}) {$\Omega_3$};
	\draw[->,very thick,blue] (1,0) -- ++(270:0.6); 
	\draw[->,very thick,blue] ({1*cos(120)},{1*sin(120)}) -- ++(210:0.6); 
	\draw[->,very thick,blue] ({1*cos(240)},{1*sin(240)}) -- ++(150:0.6); 
\end{tikzpicture}
\hspace{2cm}
\begin{tikzpicture}[ scale=0.8]
	\draw[thick,dashed] (2,0) arc[radius=2, start angle=0, end angle=360];
	\draw[very thick] (0,-2) -- (0,2);
	\draw[very thick] (-2,0) -- (2,0);
	\node at ({1.0*cos(45)},{1.0*sin(45)}) {$\Omega_1$};
	\node at ({1.0*cos(135)},{1.0*sin(135)}) {$\Omega_2$};
	\node at ({1.0*cos(225)},{1.0*sin(225)}) {$\Omega_3$};
	\node at ({1.0*cos(315)},{1.0*sin(315)}) {$\Omega_4$};
	\draw[->,very thick,blue] (0,1.2) -- ++(180:0.6); 
	\draw[->,very thick,blue] (-1.2,0) -- ++(90:0.6); 
	\draw[->,very thick,blue] (0,-1.2) -- ++(0:0.6); 
	\draw[->,very thick,blue] (1.2,0) -- ++(270:0.6); 
\end{tikzpicture}
\caption{The radial 3- and 4-partitions of the unit disk, with unit normal $\nu$ chosen as in \eqref{nu:disk}.}
\label{fig:disk}
\end{figure}

We next observe the following.

\begin{lemma}
\label{lem:alpha}
For each integer $k \geq 6$, there is a unique $\alpha \in (\frac12, \frac{k}{2})$ such that $j_{\alpha,2} = j_{k/2,1}$.
\end{lemma}

\begin{proof}
Since $\alpha \mapsto j_{\alpha,2}$ is a continuous, increasing function, we need $j_{1/2,2} < j_{k/2,1} < j_{k/2,2}$. The second inequality holds for all $k$, and the first inequality, $j_{1/2,2} < j_{k/2,1}$, holds for $k=6$ and hence for all $k \geq 6$ by monotonicity.
\end{proof}

We can now construct negative directions for the two-sided Dirichlet-to-Neumann map.

\begin{theorem}
\label{thm:disk}
Consider the radial $k$-partition for some $k \geq 6$ and choose $\alpha$ as in \Cref{lem:alpha}. For the unit normal $\nu$ defined in \eqref{nu:disk}, the function 
\begin{equation}
\label{diskef}
	f(r,\theta) = \begin{cases} J_{\alpha}\big(j_{k/2,1} r \big), & k \text{ even,} \\
		J_{\alpha}\big(j_{k/2,1} r \big)\sin\left(\tfrac{\theta}2\right), & k \text{ odd,}
	\end{cases}
\end{equation}
is contained in $\dom(\form)$ and satisfies $\form(f,f) < 0$.
\end{theorem}

In the even case $f$ is the same on every segment of $\Sigma$. This is not the case for odd $k$, though $f$ does have the same sign on every segment except $\theta=0$, where it vanishes. 
In either case, the direction of the resulting deformation depends on the sign of $f$ and the direction of $\nu$. This will be illustrated for $k=6$ after the proof.

\begin{proof}
We start with the bipartite case, where $k$ is even. Here we have $\chi_i \equiv (-1)^{i+1}$ for $1 \leq i \leq k$. To evaluate the bilinear form $\form$ in \eqref{aformdef} we thus need to solve the boundary value problem
\begin{equation}
\label{uiequationD}
	\Delta u_i + \lambda_* u_i = 0, \qquad u_i\big|_{\Sigma_i} = (-1)^{i+1} f, \qquad u_i\big|_{\pO_i \setminus \Sigma_i} = 0
\end{equation}
on each $\Omega_i$, where $\lambda_* = (j_{k/2,1})^2$. On $\Omega_1 = \{(r,\theta) : r < 1, \ 0 < \theta <  \frac{2\pi}{k} \}$ a solution is given by
\begin{equation}
	u_1(r,\theta) = J_\alpha\big(j_{k/2,1} r\big) \frac{\cos\big(\alpha(\theta - \pi/k)\big)}{\cos(\alpha\pi/k)}.
\end{equation}
A routine calculation shows this satisfies the equation $\Delta u_1 + \lambda_* u_1 = 0$. It has boundary values $u_1(r,0) = u_1(r,\frac{2\pi}k) = f(r)$ on $\Sigma_1$, and it vanishes on $\pO_1 \backslash \Sigma_1$ because $J_\alpha(j_{k/2,1}) = J_\alpha(j_{\alpha,2}) = 0$.

We then calculate the outward normal derivatives
\[
	\frac{\p u_1}{\p \nu_1}\left(r, 0\right) = \frac{\p u}{\p \nu_1}\left(r, \frac{2\pi}{k}\right) 
	= - \frac{\alpha}{r} J_\alpha\big(j_{k/2,1} r\big) \tan(\alpha\pi/k),
\]
which we use to obtain
\begin{equation}
\label{a1}
	\int_{\Omega_1} \big( |\nabla u_1|^2 - \lambda_* u_1^2\big)\,dV = \int_{\Sigma_1} u_1 \frac{\p u_1}{\p \nu_1}\,d\mu 
	= - 2\alpha \tan(\alpha\pi/k) \int_0^1 r^{-1} \big(J_\alpha(j_{k/2,1} r)\big)^2\,dr.
\end{equation}
For each $i$ the solution to \eqref{uiequationD} is the same as $u_1$, up to a sign and rotation. More precisely, $u_i(r,\theta) = (-1)^{i+1} u\big(r,\theta + 2\pi (i-1)/k\big)$. As a result, each $u_i$ has the same contribution to the sum in \eqref{aformdef}, namely \eqref{a1}, so we obtain
\begin{equation}
	\form(f,f) = - 2k \alpha \tan(\alpha\pi/k) \int_0^1 r^{-1} \big(J_\alpha(j_{k/2,1} r)\big)^2\,dr,
\end{equation}
which is negative because $\alpha < \frac{k}{2}$.

The construction in the odd case is slightly more involved, due to the angular dependence of $f$ in \eqref{diskef}. Here we make use of the spectral flow argument in \cite{BCM19}. This guarantees the existence of a unique $\sigma>0$ for which the problem $-T''(\theta) = \alpha^2 T(\theta)$ has a nontrivial solution on the interval $(0, 2\pi)$, subject to Dirichlet boundary conditions at the endpoints and jump conditions
\begin{equation}
\label{Tjump}
	T\left( \theta_i^-\right) = T\left( \theta_i^+\right), \qquad T'\left( \theta_i^+\right) - T'\left( \theta_i^-\right) = \sigma T\left( \theta_i\right)
\end{equation}
at the interior points $\theta_i = 2\pi i/k$ for $1 \leq i \leq k-1$. By \cite[Corollary 1.1]{beck2021limiting}, the solution will satisfy $T(\theta_i) = C \sin(\theta_i/2)$ for some constant $C$, so we can rescale to obtain $T(\theta_i) = \sin(\theta_i/2)$.

It follows that
\begin{equation}
	u_i(r,\theta) = (-1)^{i+1} J_\alpha\big(j_{k/2,1} r\big) T(\theta)
\end{equation}
satisfies the differential equation $\Delta u_i + \lambda_* u_i = 0$ on $\Omega_i$, with $u_i(1,\theta) = 0$ on the outer boundary. On $\Sigma_i$, where $\theta \in \{ \theta_{i-1},\theta_i\}$, we have
\[
	u_i(r,\theta) = (-1)^{i+1} J_\alpha\big(j_{k/2,1} r\big) \sin(\theta/2) = (-1)^{i+1} f(r,\theta),
\]
thus $u\big|_{\Sigma_i} = \chi_i f$. (From the choice of $\nu$ in \eqref{nu:disk} we have $\chi_i \equiv (-1)^{i+1}$ for $1 \leq i \leq k-1$; on the $k$th subdomain we have $\chi_k\big|_{\theta=0} = -1$ but this is irrelevant because $f$ vanishes when $\theta=0$.)

Calculating as in \eqref{a1}, we find 
\begin{equation}
\label{ai}
	\int_{\Omega_i} \big( |\nabla u_i|^2 - \lambda_* u_i^2\big)\,dV 
	= \Big\{ T(\theta_i) T'(\theta_i^-) - T(\theta_{i-1}) T'(\theta_{i-1}^+) \Big\}  \int_0^1 r^{-1} \big(J_\alpha(j_{k/2,1} r)\big)^2\,dr
\end{equation}
for $2 \leq i \leq k-1$. Recalling that $f$ vanishes when $\theta=0$, for $i=1$ and $i=k$ we calculate
\begin{equation}\label{ai1}
	\int_{\Omega_1} \big( |\nabla u_1|^2 - \lambda_* u_1^2\big)\,dV 
	= T(\theta_1) T'(\theta_1^-) \int_0^1 r^{-1} \big(J_\alpha(j_{k/2,1} r)\big)^2\,dr
\end{equation}
and
\begin{equation}
\label{ak}
	\int_{\Omega_k} \big( |\nabla u_k|^2 - \lambda_* u_k^2\big)\,dV 
	= - T(\theta_{k-1}) T'(\theta_{k-1}^+) \int_0^1 r^{-1} \big(J_\alpha(j_{k/2,1} r)\big)^2\,dr.
\end{equation}
Summing these and using the jump condition \eqref{Tjump}, we find that
\begin{align*}
	\sum_{i=1}^k \int_{\Omega_i} \big( |\nabla u_i|^2 - \lambda_* u_i^2\big)\,dV &=
	\left\{ \sum_{i=1}^{k-1} T(\theta_i) \big[ T'(\theta_i^-) - T'(\theta_i^+)\big] \right\} \int_0^1 r^{-1} \big(J_\alpha(j_{k/2,1} r)\big)^2\,dr \\
	&= - \sigma \left\{ \sum_{i=1}^{k-1} T(\theta_i)^2 \right\} \int_0^1 r^{-1} \big(J_\alpha(j_{k/2,1} r)\big)^2\,dr
\end{align*}
is negative, which completes the proof.
\end{proof}

We now show how the theorem can be used to study 6-partitions of the disk. In this case we have $f(r) = J_{\alpha}\big(j_{3,1} r \big)$, where $\alpha \approx 0.5657$ is the unique solution to $j_{\alpha,2} = j_{3,1}$. To get a qualitative picture of the deformation, it only matters that $f$ is positive near $r=0$ and negative near $r=1$, with a single zero in between. Recalling the choice of $\nu$, we see that the deformation will act on the nodal segments in pairs, with each being the mirror image of its neighbor, as illustrated in \Cref{disk:def}.

\begin{figure}
\begin{tikzpicture}[ scale=0.8]
	\draw[thick,dashed] (2,0) arc[radius=2, start angle=0, end angle=360];
	\draw[very thick] (0,-2) -- (0,2);
	\draw[very thick] ({-2*cos(210)},{-2*sin(210)}) -- ({2*cos(210)},{2*sin(210)});
	\draw[very thick] ({-2*cos(330)},{-2*sin(330)}) -- ({2*cos(330)},{2*sin(330)});
\end{tikzpicture}
\hspace{1cm}
\begin{tikzpicture}[ scale=0.8]
	\draw[thick,dashed] (2,0) arc[radius=2, start angle=0, end angle=360];
	\draw[very thick] (0,0) to [out=10, in=180] ({2*0.866},{2*0.5});
	\draw[very thick] (0,0) to [out=110, in=300] ({0},{2});
	\draw[very thick] (0,0) to [out=130, in=300] ({-2*0.866},{2*0.5});
	\draw[very thick] (0,0) to [out=230, in=60] ({-2*0.866},{-2*0.5});
	\draw[very thick] (0,0) to [out=250, in=60] ({0},{-2});
	\draw[very thick] (0,0) to [out=-10, in=180] ({2*0.866},{-2*0.5});
\end{tikzpicture}
\hspace{1cm}
\begin{tikzpicture}[ scale=0.8]
	\draw[thick,dashed] (2,0) arc[radius=2, start angle=0, end angle=360];
	\draw[very thick] (0,0) -- (0.2,0);
	\draw[very thick] (0.2,0) to [out=60, in=210] ({2*0.866},{2*0.5});
	\draw[very thick] (0.2,0) to [out=-60, in=-210] ({2*0.866},{-2*0.5});
	\draw[very thick] (0,0) -- ({0.2*cos(120)},{0.2*sin(120)});
	\draw[very thick] ({-0.2*0.5},{0.2*0.866}) to [out=60, in=270] (0,2);
	\draw[very thick] ({-0.2*0.5},{0.2*0.866}) to [out=180, in=330] ({-2*0.866},{2*0.5});
	\draw[very thick] (0,0) -- ({0.2*cos(240)},{0.2*sin(240)});
	\draw[very thick] ({-0.2*0.5},{-0.2*0.866}) to [out=180, in=30] ({-2*0.866},{-2*0.5});
	\draw[very thick] ({-0.2*0.5},{-0.2*0.866}) to [out=300, in=90] (0,-2);
\end{tikzpicture}
\caption{The 6-partition of the disk (left), a qualitative illustration of an energy decreasing deformation (center), and the conjectured topology of a locally minimal 6-partition (right).}
\label{disk:def}

\end{figure}

This suggests that we look for a minimal partition with the topology shown on the right of \Cref{disk:def}, so we seek an eigenfunction of a partition Laplacian whose nodal set has this structure. To do this, we fix a number $a \in (0,1)$ and consider the Laplacian on the sector $\{(r,\theta) : r < 1, \ 0 < \theta <  \frac{2\pi}{k} \}$ with Dirichlet conditions on the outer boundary and the segment $(0,a) \times \{0\}$ of the $x$-axis, and Neumann conditions on the remainder of the $x$-axis and on the segment $\theta = \pi/3$; see \Cref{fig:diskDN}. We then vary $a$ until we find an eigenfunction whose nodal set intersects the point $(a,0)$. Numerically we find that this happens for $a \approx 0.1$.  

Extending this eigenfunction to the rest of the disk by taking even reflections, we obtain an eigenfunction for the partition Laplacian whose nodal set has the desired topology, as shown in \Cref{fig:disk_cut}. Analyzing this partition Laplacian, as was done for the rectangle in  \Cref{fig:4partition}, one can verify numerically that this corresponds to the sixth eigenvalue for the partition Laplacian, so it is Courant sharp and hence a candidate for a minimal partition. Its energy is $\approx 40.7062$, whereas the radial 6-partition has $\Lambda(P_6) = (j_{3,1})^2 \approx 40.7065$. However, \cite[Figure~12]{Bogosel} shows an alternate candidate for the minimal 6-partition, which is found numerically to have energy $\approx 39.02$, suggesting that the partition we have found minimizes the energy locally but not globally.

\begin{figure}
\includegraphics[scale=0.25]{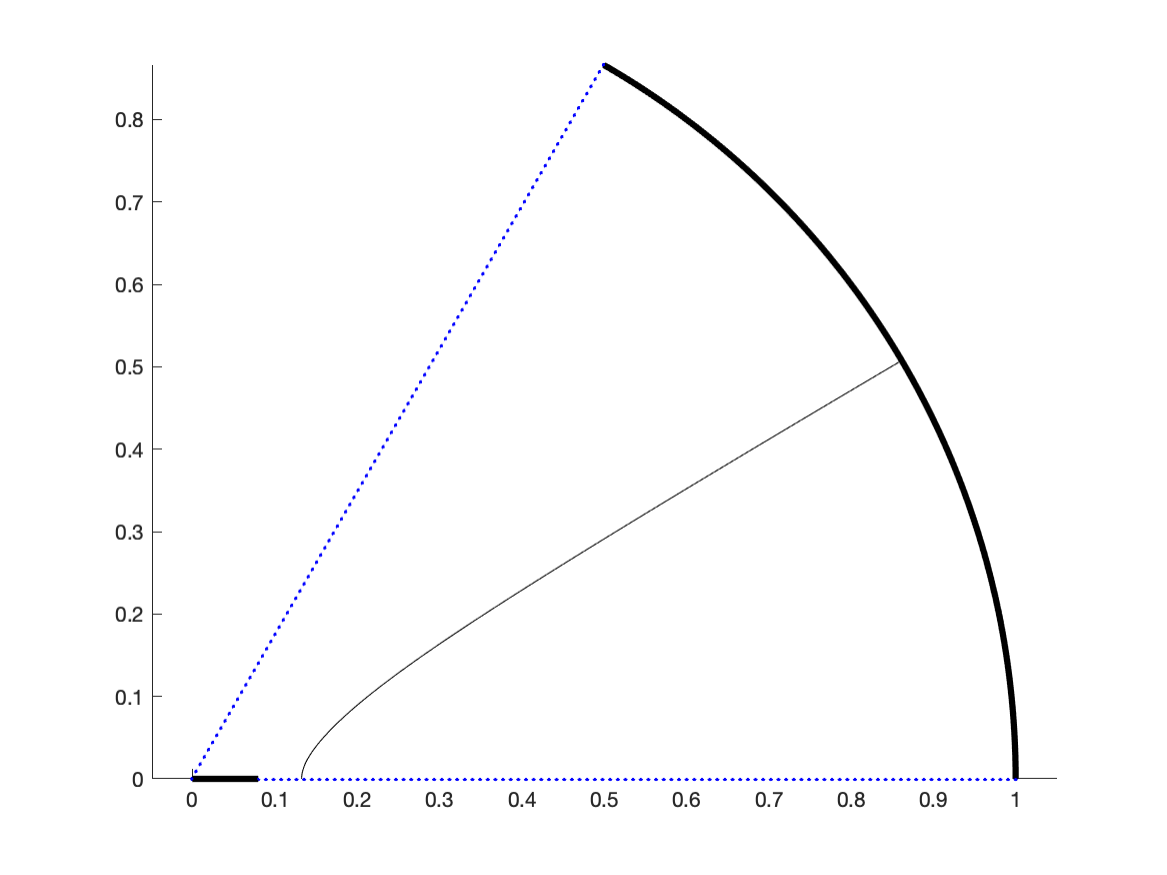}
\includegraphics[scale=0.25]{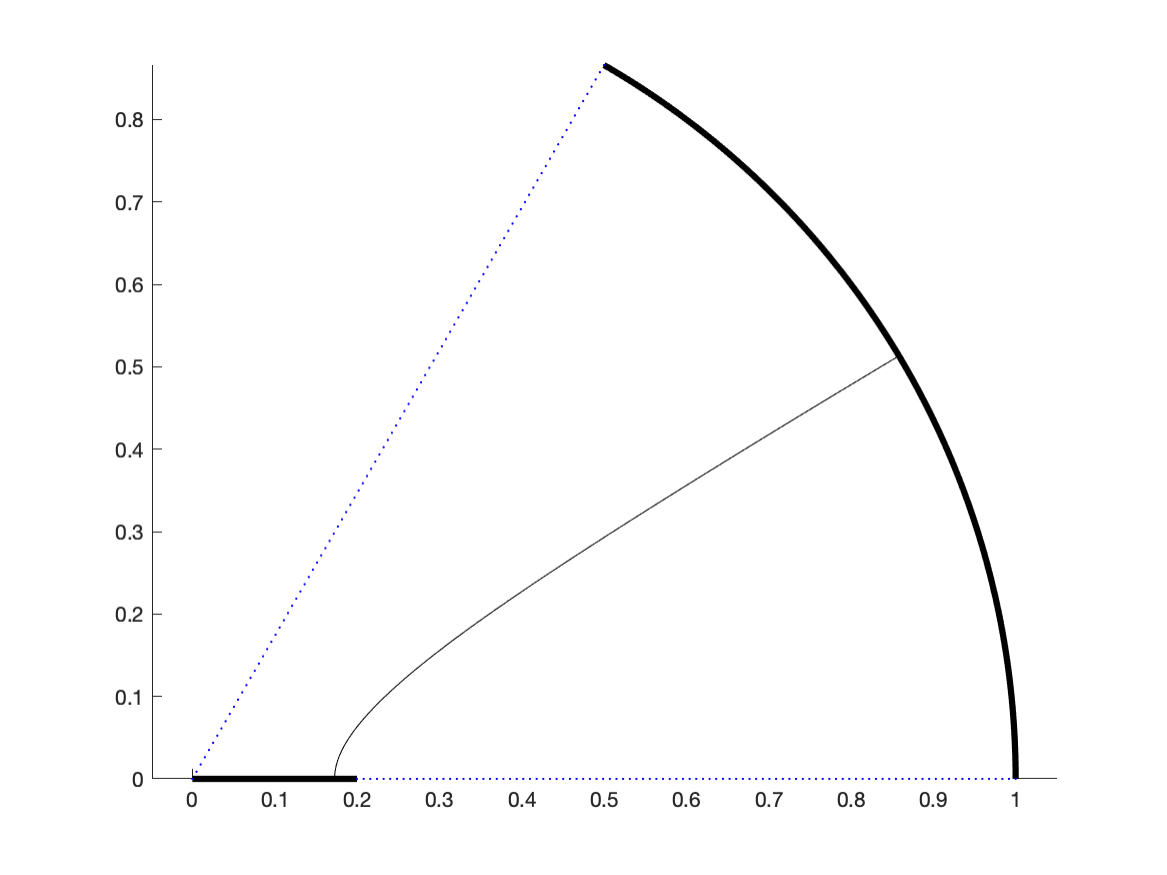} 
\includegraphics[scale=0.25]{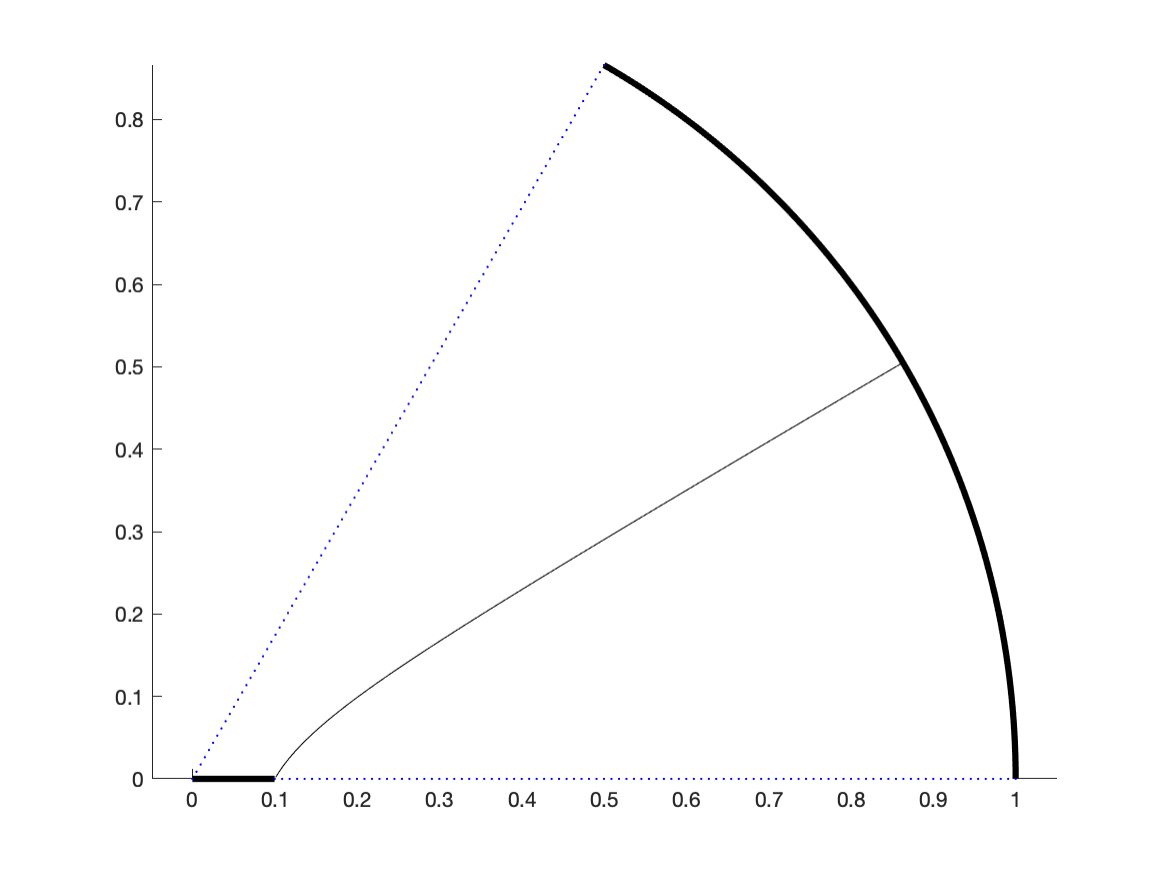}
\caption{A mixed boundary value problem on a sector of angle $\pi/3$, with Dirichlet conditions on the solid black line and Neumann conditions on the dashed blue line. The transition point (on the $x$-axis) is varied until it intersects the nodal line.}
\label{fig:diskDN}
\end{figure}

\begin{figure}
	\begin{tikzpicture}[baseline={(0,-2.25)},rotate=30]
		\draw[thick] (0,0) circle [radius=2];
		\draw[very thick, red, dotted] (0,0) -- (0,.2);
		\draw[thick] (0,.2) -- (0,2);
		
		\draw[very thick, red, dotted] (0,0) -- (-.0866*2,-.05*2);
		\draw[thick] (-.0866*2,-.05*2) -- (-.866*2,-.5*2);

		\draw[very thick, red, dotted] (0,0) -- (.0866*2,-.05*2);
		\draw[thick] (.0866*2,-.05*2) -- (.866*2,-.5*2);
	\end{tikzpicture}
\hspace{1cm} 
\includegraphics[scale=0.32]{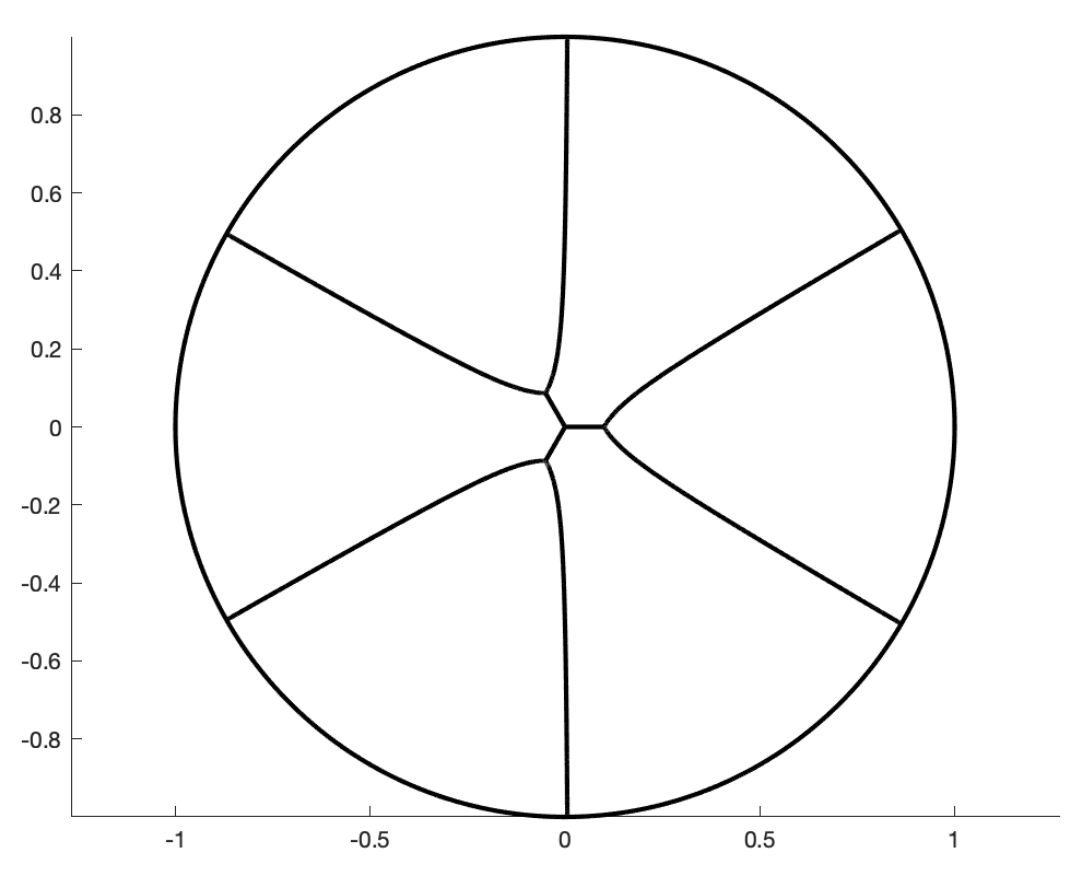}
\caption{A candidate locally minimal 6-partition of the disc (right), corresponding to a Courant-sharp eigenfunction of the operator with anti-continuity conditions imposed on the dotted red line segments (left).}
\label{fig:disk_cut}
\end{figure}
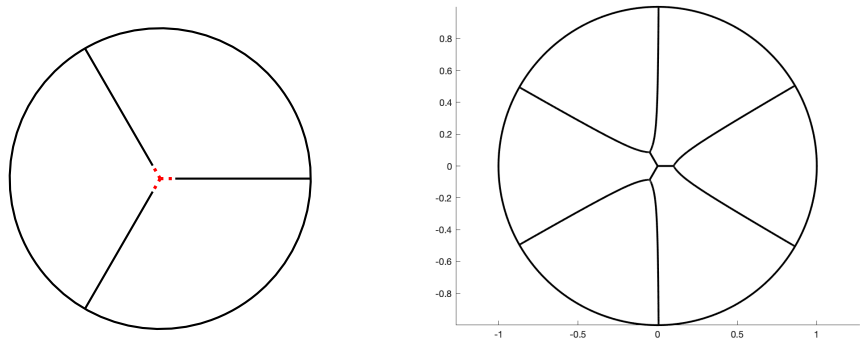

\subsection*{Acknowledgments}
G.B. acknowledges the support of NSF Grant DMS-2247473.  Y.C. was supported by NSF CAREER Grant DMS-2045494. G.C. acknowledges the support of NSERC grant RGPIN-2017-04259. P.K. acknowledges the support of NSF Grant DMS-2007408. J.L.M. acknowledges support from the NSF through  NSF Applied Math Grant DMS-2307384 and NSF FRG grant DMS-2152289.  The authors are grateful to the AIM SQuaRE program for hosting them and supporting the initiation of this project.

\bibliographystyle{plain}
\bibliography{nodal}

\def\cprime{$'$}
\begin{thebibliography}{10}

\bibitem{Adams}
Robert~A. Adams and John J.~F. Fournier.
\newblock {\em Sobolev spaces}, volume 140 of {\em Pure and Applied Mathematics
  (Amsterdam)}.
\newblock Elsevier/Academic Press, Amsterdam, second edition, 2003.

\bibitem{AloBanBer_em22}
L.~Alon, R.~Band, and G.~Berkolaiko.
\newblock Universality of nodal count distribution in large metric graphs.
\newblock {\em Exper.\ Math.}, Online First:49, 2022.

\bibitem{ACF}
Hans~Wilhelm Alt, Luis~A. Caffarelli, and Avner Friedman.
\newblock Variational problems with two phases and their free boundaries.
\newblock {\em Trans. Amer. Math. Soc.}, 282(2):431--461, 1984.

\bibitem{Azzam1}
A.~Azzam.
\newblock On {D}irichlet's problem for elliptic equations in sectionally smooth
  {$n$}-dimensional domains.
\newblock {\em SIAM J. Math. Anal.}, 11(2):248--253, 1980.

\bibitem{BanBerRazSmi_cmp12}
R.~Band, G.~Berkolaiko, H.~Raz, and U.~Smilansky.
\newblock The number of nodal domains on quantum graphs as a stability index of
  graph partitions.
\newblock {\em Commun. Math. Phys.}, 311(3):815--838, 2012.

\bibitem{BCE}
Ram Band, Graham Cox, and Sebastian~K. Egger.
\newblock Defining the spectral position of a {N}eumann domain.
\newblock {\em Anal. PDE}, 16(9):2147--2171, 2023.

\bibitem{beck2021limiting}
Thomas Beck, Isabel Bors, Grace Conte, Graham Cox, and Jeremy~L Marzuola.
\newblock Limiting eigenfunctions of {S}turm--{L}iouville operators subject to
  a spectral flow.
\newblock {\em Annales math{\'e}matiques du Qu{\'e}bec}, 45(2):249--269, 2021.

\bibitem{beck2021nodal}
Thomas Beck, Yaiza Canzani, and Jeremy~L Marzuola.
\newblock Nodal line estimates for the second {D}irichlet eigenfunction.
\newblock {\em Journal of Spectral Theory}, 11(1):323--353, 2021.

\bibitem{beck2023uniform}
Thomas Beck, Yaiza Canzani, and Jeremy~L Marzuola.
\newblock Uniform upper bounds on {C}ourant sharp {N}eumann eigenvalues of
  chain domains.
\newblock {\em arXiv preprint arXiv:2305.16452}, 2023.

\bibitem{beck2024nodal}
Thomas Beck, Marichi Gupta, and Jeremy Marzuola.
\newblock Nodal set openings on perturbed rectangular domains.
\newblock In {\em Annales Henri Poincar{\'e}}, pages 1--41. Springer, 2024.

\bibitem{BCHS}
G.~Berkolaiko, G.~Cox, B.~Helffer, and M.P. Sundqvist.
\newblock Computing nodal deficiency with a refined {D}irichlet-to-{N}eumann
  map.
\newblock {\em J. Geom. Anal.}, 32(10):Paper No. 246, 36, 2022.

\bibitem{BCCM2}
Gregory Berkolaiko, Yaiza Canzani, Graham Cox, and Jeremy~L. Marzuola.
\newblock Stability of spectral partitions and the {D}irichlet-to-{N}eumann
  map.
\newblock {\em Calc. Var. Partial Differential Equations}, 61(6):Paper No. 203,
  17, 2022.

\bibitem{BCCM3}
Gregory Berkolaiko, Yaiza Canzani, Graham Cox, and Jeremy~L. Marzuola.
\newblock Homology of spectral minimal partitions.
\newblock {\em arXiv:2406.04225}, 2024.

\bibitem{BCM19}
Gregory Berkolaiko, Graham Cox, and Jeremy~L. Marzuola.
\newblock Nodal deficiency, spectral flow, and the {D}irichlet-to-{N}eumann
  map.
\newblock {\em Lett. Math. Phys.}, 109(7):1611--1623, 2019.

\bibitem{BKS12}
Gregory Berkolaiko, Peter Kuchment, and Uzy Smilansky.
\newblock Critical partitions and nodal deficiency of billiard eigenfunctions.
\newblock {\em Geom. Funct. Anal.}, 22(6):1517--1540, 2012.

\bibitem{Bogosel}
Beniamin Bogosel and Virginie Bonnaillie-No\"el.
\newblock Minimal partitions for {$p$}-norms of eigenvalues.
\newblock {\em Interfaces Free Bound.}, 20(1):129--163, 2018.

\bibitem{bonnaillie2015nodal}
Virginie Bonnaillie-No\"{e}l and Bernard Helffer.
\newblock Nodal and spectral minimal partitions---the state of the art in 2016.
\newblock In {\em Shape optimization and spectral theory}, pages 353--397. De
  Gruyter Open, Warsaw, 2017.

\bibitem{bonnaillie2010numerical}
Virginie Bonnaillie-No{\"e}l, Bernard Helffer, and Gregory Vial.
\newblock Numerical simulations for nodal domains and spectral minimal
  partitions.
\newblock {\em ESAIM: Control, Optimisation and Calculus of Variations},
  16(1):221--246, 2010.

\bibitem{CanSar_cpam19}
Yaiza Canzani and Peter Sarnak.
\newblock Topology and nesting of the zero set components of monochromatic
  random waves.
\newblock {\em Comm. Pure Appl. Math.}, 72(2):343--374, 2019.

\bibitem{ConTerVer_am05}
Monica Conti, Susanna Terracini, and G.~Verzini.
\newblock Asymptotic estimates for the spatial segregation of competitive
  systems.
\newblock {\em Adv. Math.}, 195(2):524--560, 2005.

\bibitem{ConTerVer_cvpde05}
Monica Conti, Susanna Terracini, and Gianmaria Verzini.
\newblock On a class of optimal partition problems related to the
  {F}u\v{c}\'{\i}k spectrum and to the monotonicity formulae.
\newblock {\em Calc. Var. Partial Differential Equations}, 22(1):45--72, 2005.

\bibitem{ConTerVer_iumj05}
Monica Conti, Susanna Terracini, and Gianmaria Verzini.
\newblock A variational problem for the spatial segregation of
  reaction-diffusion systems.
\newblock {\em Indiana Univ. Math. J.}, 54(3):779--815, 2005.

\bibitem{CJM2}
Graham Cox, Christopher~K.R.T. Jones, and Jeremy~L. Marzuola.
\newblock Manifold decompositions and indices of {S}chr\"{o}dinger operators.
\newblock {\em Indiana Univ. Math. J.}, 66:1573--1602, 2017.

\bibitem{DonFef_jga92}
H.~Donnelly and C.~Fefferman.
\newblock Nodal domains and growth of harmonic functions on noncompact
  manifolds.
\newblock {\em J. Geom. Anal.}, 2(1):79--93, 1992.

\bibitem{EM1984}
David~G. Ebin and Jerrold Marsden.
\newblock Groups of diffeomorphisms and the motion of an incompressible fluid.
\newblock {\em Ann. of Math. (2)}, 92:102--163, 1970.

\bibitem{Eells}
James Eells, Jr.
\newblock On the geometry of function spaces.
\newblock In {\em Symposium internacional de topolog\'{\i}a algebraica
  {I}nternational symposium on algebraic topology}, pages 303--308. Universidad
  Nacional Aut\'{o}noma de M\'{e}xico and UNESCO, Mexico City, 1958.

\bibitem{GanMckMohSri_prep21}
Shirshendu Ganguly, Theo McKenzie, Sidhanth Mohanty, and Nikhil Srivastava.
\newblock Many nodal domains in random regular graphs.
\newblock preprint {\tt arXiv:2109.11532 [math.PR]}, 2021.

\bibitem{Geiges}
Hansj\"{o}rg Geiges.
\newblock {\em An introduction to contact topology}, volume 109 of {\em
  Cambridge Studies in Advanced Mathematics}.
\newblock Cambridge University Press, Cambridge, 2008.

\bibitem{Girault}
V.~Girault and P.-A. Raviart.
\newblock {\em Finite element approximation of the {N}avier-{S}tokes
  equations}, volume 749 of {\em Lecture Notes in Mathematics}.
\newblock Springer-Verlag, Berlin-New York, 1979.

\bibitem{G10}
P.~Grinfeld.
\newblock Hadamard's formula inside and out.
\newblock {\em J. Optim. Theory Appl.}, 146(3):654--690, 2010.

\bibitem{Gr}
Pierre Grisvard.
\newblock {\em Elliptic problems in nonsmooth domains}, volume~69 of {\em
  Classics in Applied Mathematics}.
\newblock Society for Industrial and Applied Mathematics (SIAM), Philadelphia,
  PA, 2011.

\bibitem{HHOT}
B.~Helffer, T.~Hoffmann-Ostenhof, and S.~Terracini.
\newblock Nodal domains and spectral minimal partitions.
\newblock {\em Ann. Inst. H. Poincar\'{e} Anal. Non Lin\'{e}aire},
  26(1):101--138, 2009.

\bibitem{helffer2013magnetic}
Bernard Helffer and Thomas Hoffmann-Ostenhof.
\newblock On a magnetic characterization of spectral minimal partitions.
\newblock {\em Journal of the European Mathematical Society}, 15(6):2081--2092,
  2013.

\bibitem{helffer2021spectral}
Bernard Helffer and Mikael~Persson Sundqvist.
\newblock Spectral flow for pair compatible equipartitions.
\newblock {\em Comm. Partial Differential Equations}, 47(1):169--196, 2022.

\bibitem{Henrot}
Antoine Henrot and Michel Pierre.
\newblock {\em Variation et optimisation de formes: Une analyse
  g\'{e}om\'{e}trique}, volume~48 of {\em Math\'{e}matiques \& Applications
  (Berlin)}.
\newblock Springer, Berlin, 2005.

\bibitem{JunZel_aif20}
Junehyuk Jung and Steve Zelditch.
\newblock Boundedness of the number of nodal domains for eigenfunctions of
  generic {K}aluza-{K}lein 3-folds.
\newblock {\em Ann. Inst. Fourier (Grenoble)}, 70(3):971--1027, 2020.

\bibitem{Koz06}
Vladimir Kozlov.
\newblock On the {H}adamard formula for nonsmooth domains.
\newblock {\em J. Differential Equations}, 230(2):532--555, 2006.

\bibitem{Koz12}
Vladimir Kozlov and Sergey Nazarov.
\newblock On the {H}adamard formula for second order systems in non-smooth
  domains.
\newblock {\em Comm. Partial Differential Equations}, 37(5):901--933, 2012.

\bibitem{Koz20}
Vladimir Kozlov and Johan Thim.
\newblock Hadamard asymptotics for eigenvalues of the {D}irichlet {L}aplacian.
\newblock {\em J. Math. Pures Appl. (9)}, 140:67--88, 2020.

\bibitem{lee2014multiway}
James~R Lee, Shayan~Oveis Gharan, and Luca Trevisan.
\newblock Multiway spectral partitioning and higher-order {c}heeger
  inequalities.
\newblock {\em Journal of the ACM (JACM)}, 61(6):1--30, 2014.

\bibitem{Log_am18}
Alexander Logunov.
\newblock Nodal sets of {L}aplace eigenfunctions: polynomial upper estimates of
  the {H}ausdorff measure.
\newblock {\em Ann. of Math. (2)}, 187(1):221--239, 2018.

\bibitem{LogMal_otaa18}
Alexander Logunov and Eugenia Malinnikova.
\newblock Nodal sets of {L}aplace eigenfunctions: estimates of the {H}ausdorff
  measure in dimensions two and three.
\newblock In {\em 50 years with {H}ardy spaces}, volume 261 of {\em Oper.
  Theory Adv. Appl.}, pages 333--344. Birkh\"{a}user/Springer, Cham, 2018.

\bibitem{lyons2023nodal}
Andrew Lyons.
\newblock Nodal sets of {L}aplacian eigenfunctions with an eigenvalue of
  multiplicity 2.
\newblock {\em arXiv preprint arXiv:2312.05369}, 2023.

\bibitem{M00}
William McLean.
\newblock {\em Strongly elliptic systems and boundary integral equations}.
\newblock Cambridge University Press, Cambridge, 2000.

\bibitem{miclo2015hyperboundedness}
Laurent Miclo.
\newblock On hyperboundedness and spectrum of {M}arkov operators.
\newblock {\em Inventiones mathematicae}, 200:311--343, 2015.

\bibitem{NazSod_ajm09}
Fedor Nazarov and Mikhail Sodin.
\newblock On the number of nodal domains of random spherical harmonics.
\newblock {\em Amer. J. Math.}, 131(5):1337--1357, 2009.

\bibitem{osting2014minimal}
Braxton Osting, Chris~D White, and {\'E}douard Oudet.
\newblock Minimal {D}irichlet energy partitions for graphs.
\newblock {\em SIAM Journal on Scientific Computing}, 36(4):A1635--A1651, 2014.

\bibitem{P56}
{\AA}ke Pleijel.
\newblock Remarks on {C}ourant's nodal line theorem.
\newblock {\em Comm. Pure Appl. Math.}, 9:543--550, 1956.

\end{thebibliography}

\end{document}